\documentclass[reqno,a4paper]{siamart1116}

\usepackage{amsmath,amssymb,epsfig,float,color}

\definecolor{lightblue}{rgb}{0.22,0.45,0.70}
\definecolor{darkred}{rgb}{0.82,0.15,0.20}
\definecolor{darkblue}{rgb}{0.82,0.15,0.12}

\usepackage{booktabs}

\numberwithin{figure}{section}
\numberwithin{table}{section}
\numberwithin{theorem}{section}

\newcommand\ff{\boldsymbol{f}}
\newcommand\bg{\boldsymbol{g}}

\newcommand\beps{\boldsymbol{\varepsilon}}
\newcommand\bbeta{\boldsymbol{\beta}}
\newcommand\bu{\boldsymbol{u}}
\newcommand\bv{\boldsymbol{v}}

\newcommand\cT{\mathcal{T}}

\newcommand\N{\mathbb{N}}
\newcommand\R{\mathbb{R}}

\newcommand{\norm}[1]{\ensuremath{\left\|#1\right\|}}
\newcommand*\nnorm[1]{|\!|\!| #1 |\!|\!|}

\renewcommand\O{\Omega}
\newcommand\G{\Gamma}

\newcommand\bU{\mathbf{U}}
\newcommand\bV{\mathbf{V}}
\newcommand\bW{\mathbf{W}}

\renewcommand\H{\mathrm{H}}
\renewcommand\L{\mathrm{L}}
\newcommand\Q{\mathrm{Q}}

\newcommand\LO{\L^2(\O)}
\newcommand\LOO{\L_{0}^2(\O)}
\newcommand\vdiv{\mathop{\mathrm{div}}\nolimits}

\newcommand\HCUO{\H_{0}^1(\O)}
\newcommand\HsO{\H^s(\O)}
\newcommand\HusO{\H^{1+s}(\O)}

\newcommand\bn{\boldsymbol{n}}
\newcommand\bt{\boldsymbol{t}}

\renewcommand{\theequation}{\thesection.\arabic{equation}}

\newcommand{\dx}{\,\mbox{d}x}

\newcommand\curl{\mathop{\mathbf{curl}}\nolimits}
\newcommand\rot{\mathop{\mathbf{curl}}\nolimits}
\newcommand\CE{{\mathcal E}}
\newcommand\CT{{\mathcal T}}

\renewcommand\P{\mathbb P}
\newcommand{\cred}[1]{\textcolor{red}{#1}}
\newcommand\cero{\boldsymbol{0}}

\newcommand\bomega{\boldsymbol{\omega}}
\newcommand\btheta{\boldsymbol{\theta}}
\newcommand\cP{\mathcal{P}}

\allowdisplaybreaks
\newsiamthm{remark}{Remark}

\headers{Vorticity formulation 
for the Oseen problem with variable viscosity}{V. Anaya, R. Caraballo, B. G\'omez-Vargas, 
D. Mora, R. Ruiz-Baier}

\title{Velocity-vorticity-pressure formulation 
for the Oseen problem with variable viscosity\thanks{Updated: \today. 
\funding{This work has been partially supported by DIUBB
through projects 2020127 IF/R and 194608 GI/C,
by the National Agency for Research
and Development, ANID-Chile, through project 
{\sc Centro de Modelamiento Matem\'atico} (AFB170001)
of the PIA Program: Concurso Apoyo a Centros Cient\'ificos y Tecnol\'ogicos
de Excelencia con Financiamiento Basal, by the Monash Mathematics Research Fund S05802-3951284, 
and by the HPC-Europa3 Transnational Access programme through grant HPC175QA9K.}}}
\author{V. Anaya\thanks{GIMNAP, Departamento de Matem\'atica,
Universidad del B\'io-B\'io, Concepci\'on, Chile
and CI$^2$MA, Universidad de Concepci\'on, Concepci\'on, Chile (\email{vanaya@ubiobio.cl}).} 
\and R. Caraballo\thanks{GIMNAP, Departamento de Matem\'atica,
Universidad del B\'io-B\'io, Concepci\'on, Chile 
(\email{ruben.caraballo1801@alumnos.ubiobio.cl}).}  
\and  B. G\'omez-Vargas\thanks{Secci\'on de Matem\' atica,
Sede de Occidente, Universidad de Costa Rica, San Ram\'on
de Alajuela, Costa Rica (\email{bryan.gomezvargas@ucr.ac.cr}).}
\and D. Mora\thanks{GIMNAP, Departamento de Matem\'atica,
Universidad del B\'io-B\'io, Concepci\'on, Chile
and CI$^2$MA, Universidad de Concepci\'on, Concepci\'on, Chile  (\email{dmora@ubiobio.cl}).}
\and R. Ruiz-Baier\thanks{School of Mathematical Sciences,
Monash University, 9 Rainforest Walk, Melbourne VIC 3800, Australia; and Universidad Adventista de Chile, Casilla 7-D, Chill\'an, Chile  
(\email{ricardo.ruizbaier@monash.edu}).}
} 
 
\begin{document}
\maketitle  
\begin{abstract} 
We propose and analyse an augmented mixed finite element method for the Oseen equations written in terms of velocity, vorticity, and pressure with non-constant viscosity and homogeneous Dirichlet boundary condition for the velocity. The weak formulation includes  least-squares terms arising from the constitutive equation and from the incompressibility condition, and we show that it satisfies the hypotheses of the Babu\v ska-Brezzi theory. Repeating the arguments of the continuous analysis, the stability and solvability of the discrete problem are established. The method is suited for any Stokes inf-sup stable finite element pair for velocity and pressure, while for vorticity any generic discrete space (of arbitrary order) can be used. A priori and a posteriori error estimates are derived using two specific families of discrete subspaces. Finally, we provide a set of numerical tests illustrating the behaviour of the scheme, verifying the theoretical convergence rates, and showing the performance of the adaptive algorithm guided by residual a posteriori error estimation. 
\end{abstract}



\section{Introduction}
Using vorticity as additional field in the formulation of incompressible flow 
equations can be advantageous in a number of applicative problems \cite{speziale87}. 
Starting from the seminal works \cite{duan03,DSScmame03} that focused on Stokes equations and where vorticity was sought in $\H(\curl,\Omega)$, 
several different problems including Brinkman, Navier-Stokes, and related flow problems written in terms of vorticity have been studied from the viewpoint of 
numerical analysis of  
finite volume and mixed finite element methods exhibiting diverse properties 
and specific features. Some of these contributions include 
\cite{ACPT07,AGRbCMAME2016,ACT,anaya16,cockburn,bernardi06,SScmame07,SSIMA15}. 

The starting point is the Oseen equations
in the case of variable viscosity, and written 
in terms of velocity $\bu$ and pressure $p$, as follows (see \cite{gr-1986}):
\begin{subequations}
\label{eq:oseen}
\begin{align}\label{eq:momentum0}
\sigma\bu-2\boldsymbol{\vdiv}(\nu\boldsymbol{\varepsilon}(\bu))
+(\boldsymbol{\beta}\cdot\nabla)\bu+\nabla p & = 
\ff & \mbox{ in } \Omega, \\
\vdiv\bu & =  0 & \mbox{ in } \Omega,  \label{eq:mass0}\\ 
\bu & = \boldsymbol{0}&    \mbox{ on } \Gamma,\label{eq:bc0}\\
(p,1)_{0,\Omega}&=0,&\label{eq:bc00}
\end{align}\end{subequations}
where $\sigma>0$ is inversely proportional to the time-step,
$\ff\in\LO^d$ is a force density,
$\boldsymbol{\beta}\in \H^1(\O)^d$ is the convecting velocity 
 field (not necessarily divergence-free),
and $\nu\in W^{1,\infty}(\Omega)$ is the kinematic viscosity of the fluid, 
{satisfying} 
\begin{equation}\label{nubound}
0<\nu_0\le\nu\le\nu_1.
\end{equation}
Such a set of equations will appear, for instance,
in the linearisation of non-Newtonian flow problems,
as well as in applications where viscosity may depend
on temperature, concentration or volume fractions,
or other fields where the fluid flow patterns depend on   
marked spatial distributions of viscosity
\cite{VKN2016,patil82,payne99,rudi17}. 
The specific literature related to the analysis of numerical schemes for
the Oseen equations in terms of vorticity includes the non-conforming
exponentially accurate least-squares spectral method 
proposed in \cite{mohapatra16}, least-squares methods proposed
in \cite{tsai05} for Oseen and Navier-Stokes equations with velocity
boundary conditions,
the family of vorticity-based first-order Oseen-type systems studied in \cite{chang07},
the enhanced accuracy formulation in terms of velocity-vorticity-helicity
investigated in \cite{benzi12}, and the recent mixed (exactly divergence-free)
and DG discretisations for Oseen's problem in velocity-vorticity-pressure
form  {given in} \cite{ABMRRbST2019}. 
However, in most of these references, the derivation of the variational formulations
depends on the viscosity being constant. {This is attributed to the fact that} the usual vorticity-based weak formulation results from exploiting the following identity
\begin{equation}\label{curl-curl}
\curl(\curl\bv)=-\Delta\bv+\nabla(\vdiv\bv),
\end{equation}
applied to the viscous term. 
However for a more general 
friction term of the form $-\boldsymbol{\vdiv}(\nu\boldsymbol{\varepsilon}(\bu))$,
where $\boldsymbol{\varepsilon}(\bu)$
is the strain rate tensor, the decomposition gives other additional terms that 
do not permit the direct recasting of the coupled system as done in the
 cited references above. 

Extensions to cover the case of variable viscosity
do exist in the literature. For instance, \cite{ErnM3AS98}
addresses the well-posedness of the vorticity--velocity
formulation of the Stokes problem with varying density
and viscosity, and the equivalence of the vorticity--velocity
and velocity--pressure formulations in appropriate functional
spaces is proved. More recently, in \cite{anaya22}
we have taken a different approach and employed an 
augmented vorticity-velocity-pressure formulation for Brinkman
equations with variable viscosity. 
{Here} we extend that analysis to the 
generalised Oseen equations with variable viscosity, and address in particular 
how to deal with the additional challenges posed by the
presence of the convective term that did not appear in the Brinkman 
momentum equation. 

{We will employ} the so-called augmented formulations (also known as Galerkin
least-squares methods),  {which} can be regarded as a stabilisation technique
where some terms are added to the variational formulation. Augmented finite 
elements have been considered in several
works with applications in fluid mechanics
(see, e.g., \cite{anaya13,BCGCMAME2017,bochev97,chang90,CCh16,CORbTIMA18,GOVCAMWA2020,pontaza03}  
and the references therein). These methods enjoy appealing advantages  as those described in length in, e.g., \cite{bochev09,bbf-2013}, and reformulations of the set of equations following this approach are also of great importance in the design of block preconditioners (see \cite{benzi06,farrell18} for an application in Oseen and Navier-Stokes equations in primal form, \cite{farrell20} for stress-velocity-pressure formulations for non-Newtonian flows, or \cite{campos18,farrell21} for stress-displacement-pressure mixed formulations for hyperelasticity). 
In the particular context of our mixed formulation for Oseen equations, the augmentation assists us in deriving the Babu\v{s}ka-Brezzi property of ellipticity on the kernel needed for the top-left diagonal block.  

The formulation that we employ is non-symmetric, and the augmentation terms 
appear from least-squares contributions associated with the constitutive equation and 
the incompressibility constraint. The mixed variational formulation is 
shown to be well-posed under a condition on the viscosity bounds (a generalisation of the 
usual condition needed in Oseen equations, (cf. Theorem~\ref{eyu} and Remark~\ref{CFLtipe}).
Then we establish the well-posedness of 
the discrete problem for generic inf-sup stable finite elements (for velocity 
and pressure) in combination with a generic space for vorticity approximation. 
We obtain error estimates for two stable families of finite elements. 
We also derive a reliable and efficient residual-based 
a posteriori error estimator for the mixed problem, 
which can be fully computed locally.
In summary, the advantages of the proposed method
are the possibility to obtain directly
the vorticity field
with optimal accuracy and without the need of postprocessing;
moreover, different from many existing finite element methods
with vorticity field as unknown, the present contribution
supports variable viscosity and no-slip boundary
condition in a natural way.

The contents of the paper have been structured as follows. Functional
spaces and recurrent notation is collected in the remainder of
this section. Section~\ref{sec:model} presents the governing equations
in terms of velocity, vorticity and pressure; 
we state an augmented formulation,
and we perform the solvability analysis invoking 
the Babu\v ska--Brezzi theory. The finite element discretisation
is introduced in Section~\ref{FEM:section}, where we also derive the
stability analysis and optimal error estimates for two families of stable elements.
In Section~\ref{aposte}, we
develop the a posteriori error analysis.
Several numerical tests illustrating the convergence of the proposed
method under different scenarios are reported in Section~\ref{sec:results}.

\noindent\textbf{Preliminaries.} 
Let $\O$ be a  bounded domain of $\R^d$, $d=2,3$, with Lipschitz boundary $\G=\partial\O$.
For any $s\geq 0$, the notation $\norm{\cdot}_{s,\O}$ stands
for the norm of the Hilbertian Sobolev spaces $\HsO$ or
$\HsO^d$, with the usual convention $\H^0(\O):=\LO$.


Moreover, $c$ and
$C$, with or without subscripts, tildes, or hats, will represent a generic constant
independent of the mesh parameter $h$, assuming different values
in different occurrences. In addition,
for any vector field $\bv=(v_i)_{i=1}^{3}$
and any scalar field $q$ we recall the notation: 
\[ \vdiv\bv=\sum_{i=1}^{3}\partial_i v_i, \ 
\curl\bv = 
\begin{pmatrix}
\partial_2v_3-\partial_3 v_2\\
\partial_3v_1-\partial_1v_3\\
\partial_1v_2-\partial_2v_1\end{pmatrix}, 
\ \nabla q=\begin{pmatrix} \partial_1q\\ 
\partial_2q\\
\partial_3q
\end{pmatrix},
\]
whereas for  dimension $d=2$, the curl of a vector $\bv$ and a scalar $q$ are 
 scalar function $\partial_{1}v_2-\partial_2v_1$ 
and the vector $\curl q = 
(\partial_2 q,\partial_1 q)^{t}$, respectively.

Recall that, according to \cite[Theorem~2.11]{gr-1986},
for a generic domain $\O\subseteq \R^3$, the relevant
integration by parts formula corresponds to
\begin{equation*}
\int_{\O}\curl\bomega\cdot\bv=\int_{\O}\bomega\cdot\curl\bv+\langle\bomega\times\bn,\bv\rangle_{\Gamma},
\end{equation*}
which in 2D reads as 
\begin{equation}\label{green2D}
\int_{\O}\curl\omega\cdot\bv=\int_{\O}\omega\rot\bv-
\langle\bv\cdot\bt,\omega\rangle_{\Gamma}.
\end{equation}

\section{Vorticity-based formulation}\label{sec:model}
With the aim of {proposing} 
a vorticity-based formulation  for \eqref{eq:oseen}, we consider
the following identities 
\begin{align*}
-2\boldsymbol{\vdiv}(\nu\boldsymbol{\varepsilon}(\bu))&=-2\nu\boldsymbol{\vdiv}(\boldsymbol{\varepsilon}(\bu))
-2\boldsymbol{\varepsilon}(\bu)\nabla\nu =-\nu\Delta\bu-2\boldsymbol{\varepsilon}(\bu)\nabla\nu\\
&=\nu\curl(\curl\bu)-\nu\nabla(\vdiv\bu)-2\boldsymbol{\varepsilon}(\bu)\nabla\nu.
\end{align*}
Therefore, problem \eqref{eq:oseen} rewrites as 
\begin{subequations}\label{eq:weak}
\begin{align}\label{eq:momentum1}
\sigma\bu+\nu\curl\bomega-2\boldsymbol{\varepsilon}(\bu)\nabla\nu+(\boldsymbol{\beta}\cdot\nabla)\bu+\nabla p & = 
\ff & \mbox{ in } \Omega, \\
\bomega-\curl\bu&=\cero & \mbox{ in } \Omega,\label{eq:constitutive1} \\ 
\vdiv\bu & =  0 & \mbox{ in } \Omega,  \label{eq:mass1}\\ 
\bu & = \boldsymbol{0}&    \mbox{ on } \Gamma,\label{eq:bc1}\\
(p,1)_{0,\Omega}&=0,&\label{eq:bc11}
\end{align}\end{subequations}
where we have considered the definition of the vorticity  
and have {applied the incompressibility condition}. The equations 
state, respectively, the momentum conservation, the constitutive 
relation, the mass balance, the no-slip boundary condition, and 
the pressure closure condition. 

\subsection{Variational formulation for the Oseen equations
with non-constant viscosity}

In this section, we propose a mixed variational formulation
of system \eqref{eq:momentum1}-\eqref{eq:bc11}. First,
we endow the space $\HCUO^d$ with the following norm:
\[\nnorm{\bv}_{1,\O}^2:=\Vert\bv\Vert_{0,\O}^2
+\Vert\curl\bv\Vert_{0,\O}^2+\Vert\vdiv\bv\Vert_{0,\O}^2\cred{,}\]
{and} note that for $\HCUO^d$ the above norm is equivalent
to the usual norm. In particular, we have that there
exists a positive constant $C_{pf}$ such that:
\begin{equation*}
\Vert\bv\Vert_{1,\Omega}^2\le C_{pf}(\Vert\curl\bv\Vert_{0,\Omega}^2
+\Vert\vdiv\bv\Vert_{0,\O}^2)\qquad\forall\bv\in\HCUO^d,
\end{equation*}
{where} the above inequality is a consequence of the identity
\begin{equation}\label{iderfgl}
\Vert\nabla\bv\Vert_{0,\O}^2=\Vert\curl\bv\Vert_{0,\Omega}^2
+\Vert\vdiv\bv\Vert_{0,\O}^2,
\end{equation}
which follows from \eqref{curl-curl}
and the Poincar\'e inequality. {Moreover, in order to establish} a weak formulation {for  \eqref{eq:weak}},
we will use the following identity:
\begin{equation}\label{identy}
\curl(\phi\bv)=\nabla\phi\times\bv+\phi\curl\bv,
\end{equation}
valid for any vector field $\bv$
and any scalar field $\phi$.

After testing each equation  of \eqref{eq:momentum1}-\eqref{eq:bc1}
against adequate functions, using \eqref{identy}, and imposing
the boundary conditions, we end up with the following system:
\begin{align*}
\int_{\O}(\sigma\bu+(\bbeta\cdot\nabla)\bu)\cdot\bv-2\int_{\O}\beps(\bu)\nabla\nu\cdot\bv 
+\int_{\O}\nu\bomega\cdot\curl\bv+\int_{\Omega}\bomega\cdot(\nabla\nu\times\bv)
-\int_{\O}p\vdiv\bv&=
\int_{\O}\ff\cdot\bv,\\ 
\int_{\O}\nu\btheta\cdot\curl\bu-\int_{\O}\nu\bomega\cdot\btheta&=\,0,\\
-\int_{\O}q\vdiv\bu&=\,0,
\end{align*}
for all $(\bv,\btheta,q)\in\HCUO^d\times\LO^{d(d-1)/2}\times\LOO$,
where $\LOO\,:=\,\{q\in\LO:\ (q,1)_{0,\Omega}\,=\,0\}$. 

Contrary to what is usually found in the the standard velocity-pressure mixed
formulation, the ellipticity on the kernel condition for
the Babu\v ska-Brezzi theory is not straightforward
in the above mixed formulation. Here is where the augmentation 
contributes to simplify the analysis. We introduce 
the following residual terms arising from equations
\eqref{eq:constitutive1} and \eqref{eq:mass1}:
\[
\kappa_1\int_{\O}(\curl\bu-\bomega)\cdot\curl\bv=0, \qquad
\kappa_2\int_{\O}\vdiv\bu\vdiv\bv=0\qquad\forall\bv\in\HCUO^d,
\]
where $\kappa_1$ and $\kappa_2$ are positive
parameters to be specified later on.


In this way, we propose the following augmented variational formulation
for \eqref{eq:weak}:

{\em Find $((\bu,\bomega),p)\in(\HCUO^d\times\LO^{d(d-1)/2})\times\LOO$ such that}
\begin{subequations}\label{probform2}
\begin{align}
A((\bu,\bomega),(\bv,\btheta))+B((\bv,\btheta),p)=&\;F(\bv,
\btheta)&\qquad\forall(\bv,\btheta)\in\HCUO^d\times\LO^{d(d-1)/2},\\
B((\bu,\bomega),q)=&\;0&\qquad\forall q\in\LOO,
\end{align}
\end{subequations}
where the bilinear forms and the linear functional are defined by
\begin{subequations}
\begin{align}
\label{defa1} A((\bu,\bomega),(\bv,\btheta))& := \int_{\O}(\sigma\bu+(\bbeta\cdot\nabla)\bu)\cdot\bv
+\int_{\O}\nu\bomega\cdot\btheta+\int_{\O}\nu\bomega\cdot\curl\bv-\int_{\O}\nu\btheta\cdot\curl\bu\nonumber\\
\nonumber &\quad +\kappa_1\int_{\O}\curl\bu\cdot\curl\bv+\kappa_2\int_{\O}\vdiv\bu\vdiv\bv
-\kappa_1\int_{\O}\bomega\cdot\curl\bv\nonumber\\
&\quad-2\int_{\O}\beps(\bu)\nabla\nu\cdot\bv+\int_{\O}\bomega\cdot(\nabla\nu\times\bv),\\
\label{defb1} B((\bv,\btheta),q)& :=-\int_{\O}q\vdiv\bv, \\
\label{funct}
F(\bv,\btheta)&:=\int_{\O}\ff\cdot\bv,
\end{align}
\end{subequations}
for all $(\bu,\bomega),(\bv,\btheta)\in\HCUO^d\times\LO^{d(d-1)/2}$, and
$q\in\LOO$.

As we will address in full detail in the next section,
the  augmented  mixed formulation
will permit us to analyse the problem directly
under the classical Babu\v ska-Brezzi theory~\cite{bbf-2013}.


\subsection{Well-posedness analysis}

In this section, we will address the well-posedness
of the proposed weak formulation~\eqref{probform2}.

In our analysis, we  {will need to invoke} the following inequality, which
is a consequence of the Sobolev embedding $\H^1(\O)\hookrightarrow \L^4(\O)$
\begin{equation}\label{SNS}
\left|\int_{\O}\vdiv \bbeta (\bu\cdot \bv) \right|\leq
\widehat{C}\Vert \vdiv \bbeta \Vert_{0,\O}\nnorm{\bu}_{1,\O}\nnorm{\bv}_{1,\O}.
\end{equation}

We will also make use of the following identity
(cf. \cite[Lemma~2.2]{gr-1986})
\begin{equation}\label{idiv}
\int_{\O}[(\bbeta\cdot \nabla)\bu]\cdot \bv+\int_{\O}[(\bbeta\cdot \nabla)\bv]\cdot \bu
=-\int_{\O}\vdiv \bbeta (\bu\cdot \bv).
\end{equation}

The continuity of the bilinear forms and the linear functional
(cf. \eqref{defa1}-\eqref{funct}), will be a consequence of the
following lemma, whose proof follows standard arguments in combination 
with \eqref{nubound}. 

\begin{lemma}\label{bounds}
The following estimates hold
\begin{gather*}
\left \vert \sigma\int_{\O}\bu\cdot\bv\right\vert \le
\sigma\Vert\bu\Vert_{0,\O}\Vert\bv\Vert_{0,\O},\qquad 
\left\vert\int_{\O}\nu\bomega\cdot\btheta\right\vert \le \nu_1\Vert\bomega\Vert_{0,\O}\Vert\btheta\Vert_{0,\O},\\
\left\vert\int_{\O}[(\bbeta\cdot \nabla)\bu]\cdot \bv\right\vert\le
\widehat{C}\nnorm{\bbeta}_{1,\O}\Vert \nabla \bu\Vert_{0,\O}\nnorm{\bv}_{1,\O},\\
\left\vert\int_{\O}\nu\btheta\cdot\curl\bv\right\vert \le \nu_1\Vert\btheta\Vert_{0,\O}\nnorm{\bv}_{1,\O},\qquad 
\left\vert\int_{\O}\beps(\bu)\nabla\nu\cdot\bv\right\vert \le \Vert\nabla\nu\Vert_{\infty,\O}\Vert\beps(\bu)\Vert_{0,\O}\Vert\bv\Vert_{0,\O},\\
\left\vert\int_{\O}\btheta\cdot(\nabla\nu\times\bv)\right\vert \le 2\Vert\nabla\nu\Vert_{\infty,\O}\Vert\bv\Vert_{0,\O}\Vert\btheta\Vert_{0,\O},\qquad 
\vert F(\bv,\btheta)\vert \le \Vert\ff \Vert_{0,\O}\Vert\bv\Vert_{0,\O}.
\end{gather*}
\end{lemma}

As a consequence of the above lemma, there
exist  constants $C_1, C_2, C_3 > 0$ such that
\begin{gather*}
\vert A((\bu,\bomega),(\bv,\btheta))\vert \le 
C_{1}\Vert(\bu,\bomega)\Vert\Vert(\bv,\btheta)\Vert,\qquad 
\vert B((\bv,\btheta),q)\vert \le  C_{2}
\Vert(\bv,\btheta)\Vert\Vert q\Vert_{0,\O},\\
\vert F(\bv,\btheta)\vert \le C_{3}\Vert(\bv,\btheta)\Vert,
\end{gather*}
with the product space norm defined as 
\[\Vert(\bv,\btheta)\Vert^2:=\nnorm{\bv}_{1,\O}^2+\Vert\btheta\Vert_{0,\O}^2.\]

The following lemma states  {the ellipticity of the bilinear form $A(\cdot,\cdot)$}.
\begin{lemma}\label{lem-elip}
Assume that
\begin{equation}\label{hipo}
\sigma >\frac{9\Vert\nabla\nu\Vert_{\infty,\O}^2}{\nu_0} \quad \mbox{ and }\quad
 \widehat{C}\Vert \vdiv \bbeta\Vert_{0,\Omega}<\min\left\{\sigma-\frac{9\Vert\nabla\nu\Vert_{\infty,\O}^2}{\nu_0},\dfrac{\nu_0}{12}\right\}.
\end{equation}
Then, if we choose $\kappa_1=\frac{2}{3}\nu_0$ and $\kappa_2>\dfrac{\nu_0}{3}$,
there exists a constant $\alpha>0$ such that
\[A((\bv,\btheta),(\bv,\btheta))\ge
\alpha\Vert(\bv,\btheta)\Vert^2\qquad\forall(\bv,
\btheta)\in\HCUO^d\times\LO^{d(d-1)/2}.\]
\end{lemma}

\begin{proof}
Let $(\bv,\btheta)\in\HCUO^d\times\LO^{d(d-1)/2}$.
As a consequence of Lemma~\ref{bounds}, we have that
\begin{align}\label{de1}
\nonumber \left\vert2\int_{\O}\beps(\bv)\nabla\nu\cdot\bv\right\vert
&\le 2\Vert \nabla \nu \Vert_{\infty,\O}\left(\dfrac{\nu_0}{12\Vert \nabla \nu \Vert_{\infty,\O}}\Vert \nabla \bv\Vert_{0,\O}^2+\dfrac{3\Vert \nabla \nu\Vert_{\infty,\O}}{\nu_0}\Vert \bv\Vert_{0,\O}^2\right)\\
 &=\dfrac{\nu_0}{6}(\Vert \curl \bv\Vert_{0,\O}^2+\Vert \vdiv \bv\Vert_{0,\O}^2)+\dfrac{6\Vert \nabla \nu\Vert_{\infty,\O}^2}{\nu_0}\Vert \bv\Vert_{0,\O}^2,
\end{align}
where we have used \eqref{iderfgl}. Moreover, using that
$\Vert(\nabla\nu\times\bv)\Vert_{0,\Omega}\le
2\Vert\nabla\nu\Vert_{\infty,\Omega}\Vert\bv\Vert_{0,\Omega}$, we get
\begin{align}
\nonumber \left\vert\int_{\O}\btheta\cdot(\nabla\nu\times\bv)\right\vert
&\le 2\Vert \nabla \nu\Vert_{\infty,\O}\left( \dfrac{\nu_0}{6\Vert \nabla \nu \Vert_{\infty,\O}}\Vert \btheta\Vert_{0,\O}^2+\dfrac{3\Vert \nabla \nu \Vert_{\infty,\O}}{2\nu_0}\Vert \bv \Vert_{0,\O}^2 \right)\\ 
\nonumber &=\dfrac{\nu_0}{3}\Vert \btheta\Vert_{0,\O}^2+\dfrac{3\Vert \nabla \nu \Vert_{\infty,\O}^2}{\nu_0}\Vert \bv \Vert_{0,\O}^2,\\
\nonumber \left\vert\kappa_1\int_{\O}\btheta\cdot\curl\bv\right\vert
&\le \kappa_1\left( \dfrac{\nu_0}{3\kappa_1}\Vert \btheta \Vert_{0,\O}^2+\dfrac{3\kappa_1}{4\nu_0}\Vert \curl \bv \Vert_{0,\O}^2\right)\\ \label{de3}
 &=\dfrac{\nu_0}{3}\Vert \btheta \Vert_{0,\O}^2+\dfrac{3\kappa_1^2}{4\nu_0}\Vert \curl \bv \Vert_{0,\O}^2.
\end{align}
Thus, using the Cauchy-Schwarz inequality,
\eqref{de1}-\eqref{de3}, \eqref{idiv} and \eqref{SNS}, we obtain
\begin{align*}
A((\bv,\btheta),(\bv,\btheta))\geq & \sigma\Vert \bv\Vert_{0,\O}^2+\int_{\O}[(\bbeta\cdot \nabla)\bv]\cdot \bv+\int_{\O}\nu \vert \btheta \vert^2+\kappa_1\Vert \curl \bv\Vert_{0,\O}^2+\kappa_2\Vert \vdiv \bv\Vert_{0,\O}^2\\
&-\kappa_1\int_{\O}\btheta \cdot \curl \bv-2\int_{\O}\beps(\bv)\nabla \nu\cdot \bv+\int_{\O}\btheta\cdot (\nabla \nu\times \bv)\\
\geq & \sigma\Vert \bv\Vert_{0,\O}^2-\widehat{C}\Vert \vdiv \bbeta\Vert_{0,\Omega}{\nnorm{\bv}_{1,\O}^2}+\nu_0\Vert \btheta\Vert_{0,\O}^2+\kappa_1\Vert \curl \bv\Vert_{0,\O}^2+\kappa_{2}\Vert \vdiv \bv\Vert_{0,\O}^2\\
&-\dfrac{\nu_0}{3}\Vert \btheta\Vert_{0,\O}^2-\dfrac{3\kappa_1^2}{4\nu_0}\Vert \curl \bv\Vert_{0,\O}^2-\dfrac{\nu_0}{6}(\Vert \curl \bv\Vert_{0,\O}^2+\Vert \vdiv \bv \Vert_{0,\O}^2)\\
&-\dfrac{6\Vert \nabla \nu \Vert_{\infty,\O}^2}{\nu_0}\Vert \bv \Vert_{0,\O}^2-\dfrac{\nu_0}{3}\Vert \btheta \Vert_{0,\O}^2-\dfrac{3\Vert \nabla \nu\Vert_{\infty,\O}^2}{\nu_0}\Vert \bv \Vert_{0,\O}^2\\
=&\dfrac{\nu_0}{3}\Vert \btheta \Vert_{0,\O}^2+\left(\dfrac{\nu_0}{6}-\widehat{C}\Vert \vdiv \bbeta\Vert_{0,\Omega}\right)\Vert\curl \bv \Vert_{0,\O}^2\\
&+\left(\kappa_2-\dfrac{\nu_0}{6}-\widehat{C}\Vert \vdiv \bbeta\Vert_{0,\Omega}\right)\Vert \vdiv \bv\Vert_{0,\O}^2\\
&+\left(\sigma-\dfrac{9\Vert \nabla \nu \Vert_{\infty,\O}^2}{\nu_0}-\widehat{C}\Vert \vdiv \bbeta\Vert_{0,\Omega}\right)\Vert \bv \Vert_{0,\O}^2.
\end{align*}
Now, using assumption \eqref{hipo}, we have
\begin{equation*}
A((\bv,\btheta),(\bv,\btheta))\geq \alpha\Vert(\bv,\btheta)\Vert^2,
\end{equation*}
%
where
\[
\alpha:= \min \left\{{\dfrac{\nu_0}{3}},\dfrac{\nu_0}{6}-\widehat{C}\Vert \vdiv \bbeta\Vert_{0,\Omega},\kappa_2-\dfrac{\nu_0}{6}-\widehat{C}\Vert \vdiv \bbeta\Vert_{0,\Omega},\sigma-\dfrac{9\Vert \nabla\nu \Vert_{0,\O}^2}{\nu_0} -\widehat{C}\Vert \vdiv \bbeta\Vert_{0,\Omega}\right\},
\]
which is clearly positive according to \eqref{hipo} and the assumptions on $\kappa_1$
and $\kappa_2$.
\end{proof}

Now we recall the following result related to the inf-sup condition: There exists $C>0$, depending only
on $\O$, such that (cf. \cite{G2014})
\begin{equation*}
\sup_{0\ne\bv\in\HCUO^d}\frac{\left \vert \displaystyle\int_{\O}
q\vdiv\bv\right\vert}{\Vert\bv\Vert_{1,\O}}\ge C\Vert
  q\Vert_{0,\O}\quad\forall q\in\LOO.
\end{equation*}
As a consequence, we immediately have the following lemma.
\begin{lemma}\label{inf-sup-cont}
There exists $\gamma>0$, independent of $\nu$, such that
\begin{equation*}
\sup_{0\ne(\bv,\btheta)\in\HCUO^d\times\LO^{d(d-1)/2}}\frac{\vert
  B((\bv,\btheta),q)\vert}{\Vert(\bv,\btheta)\Vert}\ge \gamma\Vert
  q\Vert_{0,\O}\quad\forall q\in\LOO.
\end{equation*}
\end{lemma}

We state the well-posedness of problem \eqref{probform2}
in the next theorem.

\begin{theorem}\label{eyu}
Assume that the hypotheses of Lemma~\ref{lem-elip} hold true.
Then, there exists a unique solution
$((\bu,\bomega),p)\in(\HCUO^d\times\LO^{d(d-1)/2})\times\LOO$
to problem \eqref{probform2}. Moreover, there exists $C>0$
such that
\[\Vert(\bu,\bomega)\Vert+\Vert p\Vert_{0,\O}\le
C\Vert\ff\Vert_{0,\O}.\]
\end{theorem}

\begin{proof}
The proof follows from Lemmas~\ref{lem-elip} and
\ref{inf-sup-cont}, and a direct consequence of the Babu\v ska-Brezzi
Theorem (\cite[Theorem~II.1.1]{bbf-2013}).
\end{proof}

\begin{remark}\label{uniquesolu}
The unique solution of problem~\eqref{probform2} also solves \eqref{eq:momentum1}-\eqref{eq:bc11}.
The equivalence follows essentially from applying integration by parts backwardly in
\eqref{probform2} and using suitable test functions. 
This  is employed in Section~\ref{aposte}
to prove the efficiency of the a posteriori error estimator.
\end{remark}

\begin{remark}\label{CFLtipe}
If 
the convective velocity  $\bbeta \in \H^1(\O)^d$ is solenoidal 
(i.e., $\vdiv\bbeta=0$ in $\Omega$), 
then problem~\eqref{probform2} is well-posed after choosing $\kappa_1=\frac{2}{3}\nu_0$, $\kappa_2>\dfrac{\nu_0}{3}$, and assuming 
\begin{equation}\label{hipo2}
\sigma\nu_0 >{9\Vert\nabla\nu\Vert_{\infty,\O}^2}.
\end{equation}
\end{remark}


\section{Numerical discretisation}
\label{FEM:section}
%
Let $\{\cT_{h}(\O)\}_{h>0}$ be a shape-regular
family of partitions of the polygonal/polyhedral region
$\bar\O$, by triangles/tetrahedrons $T$ of diameter $h_T$,
with the meshsize defined as $h:=\max\{h_T:\, T\in\cT_{h}(\O)\}$.
In what follows, given an integer $k\ge0$ and a subset
$S$ of $\R^d$, $\mathbb{P}_k(S)$ denotes the space of polynomial functions {defined on $S$ and being of degree $\leq$ $k$}.

Now, we consider generic finite dimensional subspaces
$\bV_h\subseteq\HCUO^d$, $\bW_h\subseteq\LO^{d(d-1)/2}$
and $Q_h\subseteq\LOO$ such that the following 
discrete inf-sup holds
\begin{equation}\label{inf-sup-d}
\sup_{0\ne(\bv_h,\btheta_h)\in \bV_h\times \bW_h}\frac{\vert
B((\bv_h,\btheta_h),q_h)\vert}{\Vert(\bv_h,\btheta_h)\Vert}
\ge \gamma_0\Vert q_h\Vert_{0,\O}\quad\forall q_h\in Q_h,
\end{equation}
where $\gamma_0>0$ is independent of $h$.

In this way, the above inf-sup condition can be obtained
if ($\bV_h,Q_h$) is an inf-sup stable pair for
the classical Stokes problem. Moreover,
the discrete space $\bW_h\subseteq \LO^{d(d-1)/2}$
for the vorticity can be taken as continuous
or discontinuous polynomial space. Here 
we will consider both options.

Now, we are in a position to introduce the finite element scheme
related to problem~\eqref{probform2}: 
Find $((\bu_h,\bomega_h),p_h)\in(\bV_h\times \bW_h)\times Q_h$ such that
\begin{equation}\label{probform2d}
\begin{split}
A((\bu_h,\bomega_h),(\bv_h,\btheta_h))+B((\bv_h,\btheta_h),p_h)=&\;F(\bv_h,
\btheta_h)\qquad\forall(\bv_h,\btheta_h)\in \bV_h\times \bW_h,\\
B((\bu_h,\bomega_h),q_h)=&\;0\qquad \qquad  \quad \forall q\in Q_h.
\end{split}
\end{equation}

The next step is to establish the unique solvability 
and convergence of the discrete problem \eqref{probform2d}.

\begin{theorem}\label{theorem-G}
Assume that the hypotheses of Lemma~\ref{lem-elip} hold true.
Let $\bV_h\subseteq\HCUO^d$, $\bW_h\subseteq\LO^{d(d-1)/2}$
and $Q_h\subseteq\LOO$ satisfy \eqref{inf-sup-d}.
Then, there exists a unique $((\bu_h,\bomega_h),p_h)\in(\bV_h\times \bW_h)\times Q_h$
solution to~\eqref{probform2d}.
Moreover, there exist $\hat{C}_1,\,\hat{C}_2>0$,
independent of $h$, such that
\begin{equation*}
\Vert(\bu_h,\bomega_h)\Vert+\Vert p_h\Vert_{0,\O}\le
\hat{C}_1\Vert\ff\Vert_{0,\O},
\end{equation*}
and
\begin{equation}\label{ceaest}
\begin{split}
&\Vert(\bu,\bomega)-(\bu_h,\bomega_h)\Vert
+\Vert p-p_h\Vert_{0,\O}\\
&\qquad\qquad\qquad\le\hat{C}_2\inf_{(\bv_h,\btheta_h,q_h)\in \bV_h\times \bW_h\times\Q_h}
(\nnorm{\bu-\bv_h}_{1,\O}+\Vert \bomega-\btheta_h\Vert_{0,\O}+\Vert p-q_h\Vert_{0,\O}),
\end{split}
\end{equation}
where $((\bu,\bomega),p)\in(\HCUO^d\times\LO^{d(d-1)/2})\times\LOO$ is the
unique solution of   \eqref{probform2}.
\end{theorem}

\subsection{Discrete subspaces and error estimates}

In this section, we will define explicit families of finite element subspaces
yielding the unique solvability of the discrete scheme \eqref{probform2d}.
In addition, we derive the corresponding rate of convergence
for each family.

\subsubsection{Taylor-Hood-$\mathbb{P}_{k}$}
\label{tayhood}

We start {by introducing} a family based on Taylor-Hood \cite{HT} finite
elements for velocity and pressure, and continuous
or discontinuous piecewise polynomial spaces for vorticity.
More precisely, for any $k\ge1$, we consider:
\begin{equation}\label{set1}
\begin{split}
\bV_h:&=\{\bv_h \in C(\overline{\Omega})^d:\bv_h|_{K}\in 
\mathbb{P}_{k+1}(K)^d \quad \forall K \in \mathcal{T}_h\} \cap \HCUO^d,\\
Q_h:&=\{ q_h \in C(\overline{\Omega}): q_h|_{K}\in 
\mathbb{P}_{k}(K) \quad \forall K \in \mathcal{T}_h\} \cap \LOO,\\
\bW_h^1:&=\{\btheta_h \in C(\overline{\Omega})^{d(d-1)/2}: 
\btheta_h|_{K}\in \mathbb{P}_{k}(K)^{d(d-1)/2} \quad \forall K \in \mathcal{T}_h \},\\
\bW_h^2:&=\{\btheta_h \in \LO^{d(d-1)/2}: \btheta_h|_{K}\in\mathbb{P}_{k}(K)^{d(d-1)/2} 
\quad \forall K \in \mathcal{T}_h \}.
\end{split}
\end{equation}

It is well known that $(\bV_h,Q_h)$ satisfies the inf-sup condition
\eqref{inf-sup-d}  \cite{Boffi94}.
In addition, we will consider continuous ($\bW_h^1$)
and discontinuous ($\bW_h^2$) polynomial approximations for vorticity.

Now, we recall the approximation properties of the
spaces specified in  \eqref{set1}.
Assume that $\bu \in \H^{1+s}(\O)^d$, $p \in \H^{s}(\O)$
and $\bomega \in \H^{s}(\O)^{d(d-1)/2}$, for some
$s\in(1/2,k+1]$. Then there exists $C>0$,
independent of $h$, such that
\begin{subequations}
\begin{align}
\inf_{\bv_h \in \bV_h}\nnorm{ \bu -\bv_h}_{1,\O} &\leq C h^{s}\Vert \bu\Vert_{\HusO^d},\label{Ap1}\\
\inf_{q_h \in Q_h}\Vert p -q_h \Vert_{0,\O} &\leq C h^{s}\Vert p \Vert_{\HsO},\label{Ap2}\\
\inf_{\btheta_h \in \bW_h^1}\Vert \bomega -\btheta_h \Vert_{0,\O}
&\leq C h^{s}\Vert \bomega\Vert_{\HsO^{d(d-1)/2}},\label{Ap3}\\
\inf_{\btheta_h \in \bW_h^2}\Vert \bomega -\btheta_h \Vert_{0,\O}
&\leq C h^{s}\Vert \bomega\Vert_{\HsO^{d(d-1)/2}}.\label{Ap33}
\end{align}\end{subequations}

The following theorem provides the rate of convergence of the augmented mixed
scheme \eqref{probform2d}.
\begin{theorem}\label{teoTaylorHood}
Let $k\ge1$ be an integer, and let $\bV_h,Q_h$ and $W^i_h$, $i=1,2$
be specified by \eqref{set1}.
Let $(\bu,\bomega,p)\in\HCUO^d\times\LO^{d(d-1)/2}\times\LOO$ and
$(\bu_h,\bomega_h,p_h)\in \bV_h\times \bW_h^i\times Q_h$ be the unique
solutions to the continuous and discrete problems \eqref{probform2} and
\eqref{probform2d}, respectively.  Assume that
$\bu\in\HusO^d$, $\bomega\in\HsO^{d(d-1)/2}$ and $p\in\HsO$, for some
$s\in(1/2,k+1]$. Then, there exists $\hat{C}>0$, independent of $h$, such
  that
\[\Vert(\bu,\bomega)-(\bu_h,\bomega_h)\Vert
+\Vert p-p_h\Vert_{0,\O}\leq\hat{C}h^s(\Vert\bu\Vert_{\HusO^d}
+\|\bomega\|_{\HsO^{d(d-1)/2}}+\Vert p\Vert_{\HsO}).\]
\end{theorem}

\begin{proof}
The proof follows from \eqref{ceaest} and the 
approximation properties \eqref{Ap1}-\eqref{Ap33}.
\end{proof}

\subsubsection{MINI-element-$\mathbb{P}_{k}$}
\label{minipk}
The second finite element family uses 
the so-called MINI-element for velocity and
pressure, and continuous or discontinuous
piecewise polynomials for vorticity.
Let us introduce the following spaces
(see \cite[Sections 8.6 and 8.7]{bbf-2013}, for further details):
\begin{align*}
\bU_h:&=\{\bv_h \in C(\overline{\Omega})^d:\bv_h|_{K}\in \mathbb{P}_{k}(K)^d \quad \forall K \in \mathcal{T}_h\},\\
\mathbb{B}(b_{K} \nabla H_h):&=\{\bv_{hb} \in \H^1(\O)^d:\bv_{hb}|_{K}=b_K \nabla (q_h)|_{K} \, \text{for some} \, q_h\in H_h \},
\end{align*}
where $b_{K}$ is the standard (cubic or quartic) bubble function
$\lambda_1\cdots\lambda_{d+1}\in\mathbb{P}_{d+1}(K)$,
and let us define the following finite element subspaces:
\begin{equation}\label{set2}
\begin{split}
Q_h:&=\{ q_h \in C(\overline{\Omega}): q_h|_{K}\in \mathbb{P}_{k}(K)
\quad \forall K \in \mathcal{T}_h\} \cap \LOO,\\
\bV_h:&=\bU_h \oplus \mathbb{B}(b_{K} \nabla Q_h)  \cap \HCUO^d,\\
\bW_h^1:&=\{\btheta_h \in C(\overline{\Omega})^{d(d-1)/2}:
\btheta_h|_{K}\in \mathbb{P}_{k}(K)^{d(d-1)/2} \quad \forall K \in \mathcal{T}_h \},\\
\bW_h^2:&=\{\btheta_h \in \LO^{d(d-1)/2}:\btheta_h|_{K}\in\mathbb{P}_{k}(K)^{d(d-1)/2} \quad \forall K \in \mathcal{T}_h \}.
\end{split}
\end{equation}

The rate of convergence of our augmented mixed
finite element scheme considering the above discrete
spaces \eqref{set2} is as follows.
\begin{theorem}\label{teoMINI}
Let $k\ge1$ be an integer, and let $\bV_h,Q_h$ and $W^i_h$, $i=1,2$
be given by \eqref{set2}.
Let $(\bu,\bomega,p)\in\HCUO^d\times\LO^{d(d-1)/2}\times\LOO$ and
$(\bu_h,\bomega_h,p_h)\in \bV_h\times \bW_h^i\times Q_h$ be the unique
solutions to the continuous and discrete problems \eqref{probform2} and
\eqref{probform2d}, respectively.  Assume that
$\bu\in\HusO^d$, $\bomega\in\HsO^{d(d-1)/2}$ and $p\in\HsO$, for some
$s\in(1/2,k]$. Then, there exists $\hat{C}>0$, independent of $h$, such
  that
\[\Vert(\bu,\bomega)-(\bu_h,\bomega_h)\Vert
+\Vert p-p_h\Vert_{0,\O}\leq\hat{C}h^s(\Vert\bu\Vert_{\HusO^d}
+\|\bomega\|_{\HsO^{d(d-1)/2}}+\Vert p\Vert_{\HsO}).\]
\end{theorem}

\section{A posteriori error estimator}
\label{aposte}

In this section, we propose a residual-based a posteriori
error estimator and prove its reliability and efficiency.
The analysis restricts to the two-dimensional case and using 
continuous finite element approximations for vorticity. Nevertheless, 
the extension to 3D and to discontinuous vorticity follows 
 straightforwardly.

For each $T\in\CT_h$ we let $\CE(T)$ be the set of edges of $T$, and we denote by $\CE_h$
the set of all edges in $\CT_h$, that is 
$$\CE_h=\CE_h(\O)\cup\CE_h(\G),$$
where $\CE_h(\O):=\{e\in\CE_h: e\subset\O\}$, and
$\CE_h(\G):=\{e\in\CE_h: e\subset\G\}$.
In what follows, $h_e$ stands for the diameter of a given edge
$e\in\CE_h$, $\bt_{e}=(-n_2,n_1)$, where $\bn_{e}=(n_1,n_2)$
is a fix unit normal vector of $e$.
Now, let $q\in\LO$ such that $q|_{T}\in C(T)$ for each $T\in\CT_h$,
then given $e\in\CE_h(\O)$, we denote by $[q]$ the jump of $q$ across $e$,
that is $[q]:=(q|_{T'})|_{e}-(q|_{T''})|_{e}$, where $T'$ and $T''$ are the triangles
of $\CT_{h}$ sharing the edge $e$. Moreover,
let $\bv\in\LO^2$ such that $\bv|_{T}\in C(T)^2$ for each $T\in\CT_h$.
Then, given $e\in\CE_h(\O)$, we denote by $[\bv\cdot\bt]$ the tangential
jump of $\bv$ across $e$,
that is, $[\bv\cdot\bt]:=\left((\bv|_{T'})|_{e}-(\bv|_{T''})|_{e}\right)\cdot\bt_e$,
where $T'$ and $T''$ are the triangles
of $\CT_{h}$ sharing the edge $e$.

Next, let $k\ge 1$ be an integer, and let $\bV_h, Q_h$ and $\bW_h^1$
be given as in \eqref{set1} or \eqref{set2}.
Let $(\bu,\omega,p)\in\HCUO^2\times\LO\times\LOO$ and 
$(\bu_h,\omega_h,p_h)\in \bV_h\times \bW_h^1\times Q_h$  be the unique solutions
to the continuous and discrete problems \eqref{probform2} and \eqref{probform2d}, respectively.
We introduce for each $T\in\CT_h$ the local {\it a posteriori} error indicator and its global counterpart as
\begin{align}
\Theta_{T}^2:=&h_{T}^{2}\Vert \ff -\sigma \bu_h - \nu \curl \omega_h - (\bbeta\cdot\nabla)\bu_h + 2 \beps(\bu_h)\nabla\nu - \nabla p_h\Vert_{0,T}^2 \nonumber\\
&+ \Vert \omega_h - \rot \bu_h \Vert_{0,T}^2 + \Vert\vdiv \bu_h \Vert_{0,T}^2, \qquad 
\Theta^2:=\sum_{T\in\CT_h}\Theta_{T}^2.\label{globalestimator}
\end{align}
Let us now establish reliability and efficiency of \eqref{globalestimator}. 

\subsection{Reliability}

We begin by recalling that the
continuous dependence result given in Theorem~\ref{eyu}
is equivalent to the global inf--sup condition for the
continuous formulation \eqref{probform2}. Then,
applying this estimate to the error
$(\bu-\bu_h,\omega-\omega_h,p-p_h)$,
we obtain
\begin{equation}\label{estiresidual}
\Vert(\bu,\omega)-(\bu_h,\omega_h)\Vert
+ \Vert p-p_h\Vert_{0,\Omega}\le
C_{glob}\sup_{(\bv_h,\theta_h,q_h)\in}
\frac{\mathcal{R}(\bv,\theta, q )}
{\Vert(\bv,\theta, q )\Vert},
\end{equation}
{where the residual functional $\mathcal{R}$ is defined by}
\begin{equation}\label{reliability-1}
\mathcal{R}(\bv,\theta, q )=
A((\bu-\bu_h,\omega-\omega_h),(\bv,\theta))
+B((\bv,\theta),{p-p_h})+B((\bu-\bu_h,\omega-\omega_h),q),
\end{equation}
for all $(\bv,\theta,q)\in\HCUO^2\times\LO\times\LOO$.

Some technical results are provided beforehand.
Let us first recall the  {Cl\'ement-type interpolation ope\-ra\-tor
$\mathcal{I}_h : \H^1_0(\Omega) \to Y_h$, where
$Y_h := \{v_h 	\in C(\overline{\Omega})\cap \H^1_0(\Omega): v_h\Big|_T \in
\mathbb{P}_{1}(T), \forall T \in \mathcal{T}_h\}.$}
This operator satisfies the following local
approximation properties (cf. \cite{clement75}).
\begin{lemma}\label{cle}
There exist positive constants $C_1$ and $C_2$ such that for all 
$v \in \H^1_0(\Omega)$ there hold
\begin{subequations}
\begin{align}
\Vert v - \mathcal{I}_h v \Vert_{0,T}
& \leq C_1 h_T |v|_{1,w_T} \quad \forall T \in \mathcal{T}_h,\\
\Vert v - \mathcal{I}_h v \Vert_{0,e} &\leq C_2 h^{1/2}_e |v|_{1,w_e} \quad \forall e \in \CE_h(\O),
\end{align}\end{subequations}
where $w_T:=\bigcup\{T' \in \mathcal{T}_h: T' \cap T \neq \emptyset\}$
and $w_e:=\bigcup\{T' \in \mathcal{T}_h: T' \cap e \neq \emptyset\}$.
\end{lemma}

The main result of this section is stated as follows.
\begin{theorem}\label{th:reliability}
There exists a positive constant $C_{\mathrm{rel}}$,
independent of  $h$, such that 
\begin{equation}\label{relia}
\Vert(\bu,\omega)-(\bu_h,\omega_h)\Vert
+ \Vert p-p_h\Vert_{0,\Omega}  \leq {C}_{\mathrm{rel}} \;\Theta.
\end{equation}
\end{theorem}
\begin{proof}
From \eqref{reliability-1} and the continuous
problem~\eqref{probform2}, we have that,
\begin{align*}
\mathcal{R}(\bv,\theta, q )
& = \int_{\O}\ff\cdot\bv-\Big(A((\bu_h,\omega_h),(\bv,\theta))
+B((\bv,\theta),p_h)+B((\bu_h,\omega_h),q)\Big)\\
&=\int_{\O}\big(\ff-\sigma \bu_h-(\bbeta\cdot \nabla)\bu_h+2\beps(\bu_h)\nabla \nu\big)\cdot \bv
-\int_{\O}\nu (\omega_h-\rot \bu_h)\theta\\
&\quad-\kappa_1 \int_{\O}(\rot \bu_h-\omega_h)\rot\bv -\kappa_2 \int_{\O} \vdiv \bu_h \vdiv \bv\\
&\quad-\left( \int_{\O}\nu \omega_h\rot \bv+\int_{\O}\omega_h(\nabla\nu\times\bv)\right)
+\int_{\O}p_h\vdiv \bv+\int_{\O}q\vdiv \bu_h.
\end{align*}
Using the identity $\rot(\nu\bv)=\nabla\nu\times\bv+\nu\rot\bv$
and integration by parts on the above residual (cf. \eqref{green2D}), we obtain
\begin{align*}
\mathcal{R}(\bv,\theta, q)
&=\int_{\O}(\ff-\sigma \bu_h-(\bbeta\cdot \nabla)\bu_h+2\beps(\bu_h)\nabla \nu)\cdot \bv
-\int_{\O}\nu (\omega_h-\rot \bu_h)\theta\\
&-\kappa_1 \int_{\O}(\rot \bu_h-\omega_h)\rot \bv 
-\kappa_2 \int_{\O} \vdiv \bu_h \vdiv \bv+\int_{\O}q\vdiv \bu_h\\
&-\sum_{T\in \mathcal{T}_h}\left( \int_{T}\nu \curl \omega_h \cdot \bv
-\langle \bv\cdot\bt,\nu\omega_h \rangle_{\partial T} -\int_{T}\nabla p_h\cdot \bv+\langle \bv\cdot \bn, p_h\rangle_{\partial T}\right)\\
&=\sum_{T\in \mathcal{T}_h}\int_{T}(\ff-\sigma \bu_h-\nu \curl \omega_h-(\bbeta\cdot \nabla)\bu_h+2\beps(\bu_h)\nabla \nu-\nabla p_h)\cdot \bv\\
& -\int_{\O}\nu (\omega_h-\rot \bu_h)\theta-\kappa_1 \int_{\O}(\rot\bu_h-\omega_h)\rot \bv
-\kappa_2 \int_{\O} \vdiv \bu_h \vdiv \bv +\int_{\O}q\vdiv \bu_h,
\end{align*}
where we have used the fact  {that} $\omega_h$ and $p_h$
are piecewise continuous functions.
Hence, since from \eqref{reliability-1} we have $\mathcal{R}(\bv_h,\theta_h,q_h)=0$,
we obtain
\begin{align*}
\mathcal{R}(\bv,\theta, q)&=\mathcal{R}(\bv-\bv_h,\theta-\theta_h, q-q_h)\\
&=\sum_{T\in \mathcal{T}_h}\int_{T}\big(\ff-\sigma \bu_h-\nu \curl \omega_h
-(\bbeta\cdot \nabla)\bu_h+2\beps(\bu_h)\nabla \nu-\nabla p_h\big)\cdot (\bv-\bv_h)\\
&\quad -\int_{\O}\nu (\omega_h-\rot \bu_h)(\theta-\theta_h)
-\kappa_1 \int_{\O}(\rot\bu_h-\omega_h)\rot (\bv-\bv_h)\\
&\quad-\kappa_2 \int_{\O} \vdiv \bu_h \vdiv (\bv-\bv_h)+\int_{\O}(q-q_h)\vdiv \bu_h.
\end{align*}

{Thus,} it suffices to take $\bv_h:=\mathcal{I}_h(\bv)$ (cf. Lemma~\ref{cle}),
and $\theta_h:=\Pi(\theta)$ and $q_h:=\Pi(q)$ with
$\Pi$ being the $L^2$-projection onto  piecewise constants. And then, using the 
Cauchy-Schwarz inequality, triangle inequality, properties
for $\mathcal{I}_h$ given by Lemma~\ref{cle}  and \cite[Lemma 1.127]{EG2004},
and approximation properties for $\Pi$, we  obtain
\begin{align*}
\mathcal{R}(\bv,\theta, q)&\leq C_1 \sum_{T\in \mathcal{T}_h}h_T\Vert \ff
-\sigma \bu_h-\nu \curl \omega_h-(\bbeta\cdot \nabla)\bu_h
+2\beps(\bu_h)\nabla \nu-\nabla p_h\Vert_{0,T}\vert \bv\vert_{1,w_T}\\
&\quad +\sum_{T\in \mathcal{T}_h} (\nu_1+\kappa_1)\Vert \omega_h
-\rot \bu_h\Vert_{0,T} (C_3\Vert \theta\Vert_{0,T}+ \vert \bv-\bv_h \vert_{1,T}) \\
&\quad+\sum_{T\in \mathcal{T}_h}(\kappa_2+1) \Vert \vdiv \bu_h \Vert_{0,T}(\vert \bv-\bv_h\vert_{1,T}+C_4\Vert q\Vert_{0,T})\\
&\leq  \widehat{C}_1 \left( \sum_{T\in \mathcal{T}_h}h_T^2\Vert f-\sigma \bu_h-\nu \curl\omega_h-(\bbeta\cdot \nabla)\bu_h+2\beps(\bu_h)\nabla \nu-\nabla p_h\Vert_{0,T}^2\right)^{1/2}\!\!\Vert \bv\Vert_{1,\O}\\
&\quad +\widehat{C}_2\left( \sum_{T\in \mathcal{T}_h}\Vert\omega_h-\rot\bu_h\Vert_{0,T}^2\right)^{1/2} (\Vert \theta\Vert_{0,\O}+ {\Vert \bv\Vert_{1,\O}}) \\
&\quad +\widehat{C}_3\left( \sum_{T\in \mathcal{T}_h} \Vert \vdiv \bu_h \Vert_{0,T}^2\right)^{1/2}(\Vert \bv\Vert_{1,\O}+\Vert q\Vert_{0,\O}).
\end{align*}
And the proof of \eqref{relia} follows from \eqref{estiresidual}
and the above estimate. 
\end{proof}

\subsection{Efficiency}
This subsection deals with the efficiency of the  {\it a posteriori}
error estimator. For simplicity, we  {will} assume that
the given convective velocity $\bbeta$ and the viscosity $\nu$
are polynomial functions both of degree $s$.
The general case can be proved by repeating the same arguments
and requiring an additional regularity for the data.


A major role in the proof of efficiency is played by element and edge bubbles (locally supported non-negative functions), whose definition we recall in what follows. 
For $T\in\mathcal{T}_{h}(\O)$ and $e\in \CE(T)$, let $\psi_T$ and $\psi_e$, respectively, be the interior and edge bubble functions defined as in, e.g., \cite{ain-ode}. Let $\psi_T\in \mathbb{P}_3(T)$ with $supp (\psi_T) \subset T,$  $\psi_T=0$ on $\partial T$ and
$0\leq \psi_T \leq 1$ in $T.$ Moreover,  let $\psi_e |_T \in \mathbb{P}_2(T)$ with  $supp (\psi_e) \subset \O_e := \{T' \in \mathcal{T}_h(\O):
e\in \CE(T')\},$ $\psi_e =0$ on $\partial T\setminus  e,$ and $0\leq \psi_e \leq 1 $ in $\O_e.$  Again, let us recall an extension  operator $E: C^0(e) \mapsto C^0(T)$ that satisfies $E(q)\in \mathbb{P}_k(T)$ and $E(q)|_e = q$ for all $q\in \mathbb{P}_k(e)$ and for all 
$k\in \N\cup \{0\}.$

We now summarise  the properties of $\psi_T, \psi_e$ and $E$
in the following lemma (see \cite{ain-ode,verfuerth96}).
\begin{lemma}\label{lem:psi}
The following properties hold:
\begin{itemize}
\item [(i)] For $T\in \mathcal{T}_h$ and for $v\in \mathbb{P}_k(T)$, there is a positive constant $C_1$  such that 
\begin{align*}
C_1^{-1}\; \|v\|^2_{0,T} \leq \int_{T} \psi_T v^2 \dx \leq C_1 \|v\|^2_{0,T},\qquad 
C_1^{-1}\; \|v\|^2_{0,T} \leq \|\psi v\|^2_{0,T} + h_T^2  |\psi v|^2_{1,T} \leq C_1 \|v\|^2_{0,T}. 
\end{align*}
\item [(ii)] For $e\in \CE_h$ and $v\in \mathbb{P}_k(e)$, there  exists a positive constant, say $C_1$, such that
\[C_1^{-1}\; \|v\|^2_{0,e} \leq \int_{e} \psi_e v^2 ds \leq C_1 \|v\|^2_{0,e} .\]
\item [(iii)] For $T\in \mathcal{T}_h$, $e\in \CE(T)$ and  for $v\in \mathbb{P}_k(e)$, there is a positive constant, 
again say $C_1$, 
such that
\[ \| \psi_e^{1/2}  \;E(v) \|^2_{0,T}  \leq C_1 h_e\; \|v\|^2_{0,e} .\]
\end{itemize}
\end{lemma}

The following classical result which states
an inverse estimate will also be used.
\begin{lemma}\label{inverse}
Let $k, l, m\in\N\cup\{0\}$ such that $l\le m$. Then, there exists $\tilde{C}>0$, depending only on
$k,l,m$ and the shape regularity of the triangulations, such that for each triangle $T$
there holds
\begin{equation*}
\vert q\vert_{m,T}\le \tilde{C} h_{T}^{l-m}\vert q\vert_{l,T}\quad\forall q\in\P_{k}(T).
\end{equation*}
\end{lemma}

In order to prove the efficiency of the a posteriori error estimator, we will bound
each term defining $\Theta_T$ in terms of local errors.

\begin{theorem} \label{thm:efficiency}
There is a positive constant $C_{\mathrm{eff}}$, independent of $h$, such that 
\[
C_{\mathrm{eff}} \;\Theta \leq \;\Vert(\bu,\omega)
-(\bu_h,\omega_h)\Vert
+ \Vert p-p_h\Vert_{0,\Omega} + \mathrm{ h.o.t.},
\]
where  $\mathrm{ h.o.t.}$  denotes higher-order terms.
\end{theorem}

\begin{proof}
Using that $\omega - \rot \bu=0$ and $\vdiv\bu=0$ in $\Omega$
(see \eqref{eq:constitutive1} and \eqref{eq:mass1}, respectively), we immediately have that
\[\Vert \omega_h - \rot \bu_h \Vert_{0,T} + \Vert\vdiv \bu_h \Vert_{0,T}\le
\Vert\rot(\bu-\bu_h)\Vert_{0,T}+\Vert\vdiv(\bu-\bu_h)\Vert_{0,T}
+\Vert\omega-\omega_h\Vert_{0,T}.\]

On the other hand, with the help of the $\L^2(T)^2$-orthogonal
projection $\cP_{T}^{\ell}$ onto $\mathbb{P}_{\ell}(T)^2$, for $\ell\geq (s+k+1),$
with respect to the weighted $\L^2$-inner product $(\psi_T \ff,\bg)_{0,T}$, for 
$\ff,\bg\in \L^2(T)^2,$ it now follows that
\begin{align*}
\Vert \ff -\sigma \bu_h - & \nu \curl \omega_h - (\bbeta\cdot\nabla)\bu_h
+ 2 \beps(\bu_h)\nabla\nu - \nabla p_h\Vert_{0,T}^2 \\
 &=\Vert \ff - \cP_{T}^{\ell}(\ff) + \cP_{T}^{\ell}(\ff) 
 -\sigma \bu_h - \nu \curl \omega_h - (\bbeta\cdot\nabla)\bu_h
 + 2 \beps(\bu_h)\nabla\nu - \nabla p_h\Vert_{0,T}^2\\
&\leq \Vert \ff - \cP_{T}^{\ell}(\ff) \Vert_{0,T}^2
+ \Vert \cP_{T}^{\ell}(\ff)  -\sigma \bu_h - \nu \curl \omega_h - (\bbeta\cdot\nabla)\bu_h + 2 \beps(\bu_h)\nabla\nu - \nabla p_h\Vert_{0,T}^2\\
&= \Vert \ff - \cP_{T}^{\ell}(\ff) \Vert_{0,T}^2 + \Vert \cP_{T}^{\ell}(\ff 
-\sigma \bu_h - \nu \curl \omega_h - (\bbeta\cdot\nabla)\bu_h 
+ 2 \beps(\bu_h)\nabla\nu - \nabla p_h)\Vert_{0,T}^2.
\end{align*}
For the second term on the right-hand side, an {application} of Lemma~\ref{lem:psi} shows that
\begin{align*}
\Vert \cP_{T}^{\ell}(\ff -\sigma \bu_h & - \nu \curl \omega_h - (\bbeta\cdot\nabla)\bu_h +  2 \beps(\bu_h)\nabla\nu - \nabla p_h)\Vert_{0,T}^2 \\
& \leq \Vert \psi_T^{1/2} \cP_{T}^{\ell}(\ff -\sigma \bu_h - \nu \curl \omega_h - (\bbeta\cdot\nabla)\bu_h + 2 \beps(\bu_h)\nabla\nu - \nabla p_h)\Vert_{0,T}^2\\
&=\int_{T}\psi_T\cP_{T}^{\ell}(\ff -\sigma \bu_h - \nu \curl \omega_h - (\bbeta\cdot\nabla)\bu_h + 2 \beps(\bu_h)\nabla\nu - \nabla p_h)\\
&\qquad\times(\ff -\sigma \bu_h - \nu \curl \omega_h - (\bbeta\cdot\nabla)\bu_h
+ 2 \beps(\bu_h)\nabla\nu - \nabla p_h),
\end{align*}
where we have used the fact that $\cP_{T}^{\ell}$ is the $L^2(T)^2$-orthogonal projection.
{Thus, from the above inequality, and \eqref{eq:momentum1} (cf. Remark~\ref{uniquesolu}), we can deduce that}
\begin{align*}
\Vert \cP_{T}^{\ell}(\ff &-\sigma \bu_h  - \nu \curl \omega_h - (\bbeta\cdot\nabla)\bu_h +  2 \beps(\bu_h)\nabla\nu - \nabla p_h)\Vert_{0,T}^2 \\
&{\leq}\int_{T}\psi_T\cP_{T}^{\ell}(\ff -\sigma \bu_h - \nu \curl \omega_h - (\bbeta\cdot\nabla)\bu_h + 2 \beps(\bu_h)\nabla\nu - \nabla p_h)\\
&\quad\times(\sigma (\bu-\bu_h) + \nu\curl (\omega-\omega_h) + (\bbeta\cdot\nabla)(\bu-\bu_h)
- 2 \beps(\bu-\bu_h)\nabla\nu + \nabla (p-p_h)).
\end{align*}
Next, using that the viscosity is a polynomial function,
the bound follows by integration by parts on the
terms $\curl (\omega-\omega_h)$ and $\nabla (p-p_h)$,
Cauchy-Schwarz inequality and an inverse inequality (cf. Lemma~\ref{inverse}).
We end the proof by observing that the required efficiency
bound follows straightforwardly from
the estimates above, and after assuming additional regularity for
 $\ff$.
\end{proof}

\section{Numerical results}\label{sec:results}

In this section, we present some numerical experiments carried out
with the schemes proposed and analysed in Section~\ref{FEM:section}.
We also present two numerical examples in $\mathbb{R}^2$,
confirming the reliability and efficiency of the a posteriori error estimator
$\Theta$ derived in Section~\ref{aposte}, and showing the behaviour
of the associated adaptive algorithm.
The solution of all linear systems is carried out with the multifrontal
massively parallel sparse direct solver MUMPS.

We construct a series of uniformly successively
refined triangular meshes for $\O$ and compute  individual 
errors 
\[e(\bu) =\nnorm{\bu-\bu_h}_{1,\O},\quad  
e(\bomega) =\|\bomega-\bomega_h\|_{0,\O},\quad  
e(p)=\|p-p_h\|_{0,\O},\]
and convergence rates
\begin{equation}\label{rates}
r(\bu)=\dfrac{\log(e(\bu)/\widehat{e}(\bu))}{\log(h/\hat{h})},
\qquad r(\bomega)=\dfrac{\log(e(\bomega)/\widehat{e}(\bomega))}{\log(h/\hat{h})},\qquad
r(p)=\dfrac{\log(e(p)/\widehat{e}(p))}{\log(h/\hat{h})},
\end{equation}
where $e,\widehat{e}$ denote errors generated on two consecutive
meshes of sizes $h,\hat{h}$, respectively.

\subsection{Example 1: Convergence test using manufactured solutions} 
The first test consists of approximating closed-form solutions 
on a two-dimensional domain $\Omega=(0,1)^2$.
We construct the forcing term $\ff$ so that the exact
solution to \eqref{eq:momentum1}-\eqref{eq:mass1}
is given by the following smooth functions
\begin{gather*}
p(x,y):=\left(\left(x-\dfrac{1}{2}\right)^3y^2+(1-x)^3\left(y-\dfrac{1}{2}\right)^3 \right),\\
\bu(x,y):=\curl(1000x^2(1-x)^4y^3(1-y)^2), \qquad 
 \omega(x,y):=\rot \bu,\end{gather*}
which satisfy the incompressibility constraint as well as the
boundary conditions. In addition, we take
$\bbeta =\bu,$ 
and two specifications for the variable viscosity are considered, 
\[\nu_{a}(x,y)=\nu_0+(\nu_1-\nu_0)xy, \  
\nu_{b}(x,y)=\nu_0+(\nu_1-\nu_0)exp(-10^{13}((x-0{\text.}5)^{10}+(y-0{\text.}5)^{10})),\]
with $\nu_0=0{\text .}001$, $\nu_1=1$, and taking 
$\kappa_1=\frac{2}{3}\nu_0$, $\kappa_2=\frac{\nu_0}{2}$ and $\sigma=100$.
The error history of the method introduced
in Section~\ref{tayhood} with discontinuous
finite elements for vorticity $(\bW_h^2)$ for $k=1$
and for the two different viscosity functions
is collected in Tables~\ref{ta1}
and \ref{ta2}, respectively.
These values indicate optimal accuracy $O(h^2)$ for $k=1$, and
for $\nu_a$ and $\nu_b$, according to Theorem~\ref{teoTaylorHood}.

\begin{table}[!t]
\begin{center}
{\small\begin{tabular}{@{}cccccccc@{}}
\toprule
$h$&$\nnorm{\bu-\bu_h}_{1,\O}$& $r(\bu)$ &$\Vert \omega-\omega_h\Vert_{0,\O}$
&$r(\omega)$& $\Vert p-p_h\Vert_{0,\O}$ & $r(p)$\\ \midrule
0.7071 & 10.86 &    -- & 9.1110 &    -- & 2.5470 &   --\\
0.3536 & 4.4240 & 1.3 & 3.5500 & 1.4 & 1.5330 & 0.7\\
0.1768 & 1.2540 & 1.8 & 0.9854 & 1.9 & 0.3493 & 2.1\\
0.0883 & 0.3492 & 1.8 & 0.2470 & 2.0 & 0.0622 & 2.4\\
0.0441 & 0.1096 & 1.7 & 0.0613 & 2.0 & 0.0107 & 2.5\\
0.0221 & 0.0327 & 1.8 & 0.0151 & 2.0 & 0.0020 & 2.4\\
0.0110 & 0.0075 & 2.1 & 0.0037 & 2.0 & 0.0004 & 2.2\\ \bottomrule
\end{tabular}}
\caption{Example~1: convergence tests against analytical solutions 
on a sequence of uniformly refined triangulations of the domain 
$\O$ and the viscosity function $\nu_a$.}
\label{ta1}
\end{center}
\end{table} 
 
 \begin{table}[!t]
\begin{center}
{\small\begin{tabular}{@{}ccccccc@{}}
\toprule
$h$&$\nnorm{\bu-\bu_h}_{1,\O}$& $r(\bu)$ &$\Vert \omega-\omega_h\Vert_{0,\O}$
&$r(\omega)$& $\Vert p-p_h\Vert_{0,\O}$ & $r(p)$\\ \midrule
0.7071 & 10.91 &   -- & 9.1340 &   -- & 2.1190 &   --\\
0.3536 & 4.489 & 1.3 & 3.6710 & 1.3 & 1.4580 & 0.5\\
0.1768 & 1.367 & 1.7 & 1.1200 & 1.7 & 0.2789 & 2.4\\
0.0883 & 0.366 & 1.9 & 0.2951 & 1.9 & 0.0482 & 2.5\\
0.0441 & 0.113 & 1.7 & 0.0864 & 1.8 & 0.0070 & 2.8\\
0.0221 & 0.036 & 1.6 & 0.0220 & 2.0 & 0.0014 & 2.3\\
0.0110 & 0.007 & 2.1 & 0.0046 & 2.2 & 0.0003 & 2.2\\ \bottomrule
\end{tabular}}
\caption{Example~1: convergence tests against analytical solutions 
on a sequence of uniformly refined triangulations of the domain 
$\O$ and the viscosity function $\nu_b$.}
\label{ta2}
\end{center}
\end{table}

\subsection{Example 2: Convergence in 3D} 
The aim of this numerical test is to assess the accuracy 
of the method in the 3D case. With this end,
we consider the domain $\O:=(0,1)^3$
and take $\ff$ so that the
exact solution is given by
\begin{gather*}p(x,y,z):=1-x^2-y^2-z^2,\quad 
\varphi(x,y,z):=
x^2(1-x)^2y^2(1-y)^2z^2(1-z)^2,\\
\bu(x,y,z)=\curl \varphi,\qquad 
\bomega(x,y,z)=\curl\bu,\end{gather*}
and we   consider 
$\bbeta=\bu$, and $\nu_c(x,y,z)=\nu_0+(\nu_1-\nu_0)x^{2}y^{2}z^{2}$.
The remaining constants are $\nu_0=0{\text .}1$, $\nu_1=1$, $\kappa_1=\frac{2}{3}\nu_0$,
$\kappa_2=\frac{\nu_0}{2}$, and $\sigma=1000$.
We observe that the hypothesis of Lemma~\ref{lem-elip} are satisfied.  Additionally, we  employ  finite elements
with $k=1$, {that is,
$\bV_h$ approximating the velocity,} and piecewise
linear and continuous elements for vorticity and pressure.

In Table~\ref{taexam2}, we summarise the convergence history
for a sequence of uniform meshes. For velocity we observe the $O(h)$ convergence 
predicted by Theorem~\ref{teoMINI}, whereas the approximation of 
vorticity and pressure seem to be superconvergent. 
Figure~\ref{vpvfigures3D} 
displays velocity and vorticity streamlines as well as 
the approximate pressure distribution.

\begin{table}{\!t}
\begin{center}
{\small\begin{tabular}{@{}ccccccc@{}}
\toprule
$h$&$\nnorm{\bu-\bu_h}_{1,\O}$& $r(\bu)$ &$\Vert \bomega-\bomega_h\Vert_{0,\O}$
&$r(\bomega)$& $\Vert p-p_h\Vert_{0,\O}$ & $r(p)$\\ \midrule
0.866  & 0.01021 &  -- & 0.00299 &  -- & 0.04732 &  -- \\
0.433  & 0.00858 & 0.3 & 0.00125 &  1.3 & 0.01399 &  1.8 \\
0.288  & 0.00665 & 0.6 & 0.00067 &  1.5 & 0.00572 &  2.2 \\
0.216  & 0.00513 &  0.9 & 0.00043 &  1.5 & 0.00290 &  2.4 \\
0.173  & 0.00398 &  1.1 & 0.00030 &  1.5 & 0.00171 &  2.4 \\
0.144  & 0.00313 &  1.3 & 0.00023 &  1.5 & 0.00112 &  2.3 \\
0.123  & 0.00251 &  1.3 & 0.00018 &  1.5 & 0.00079 &  2.2 \\ \bottomrule
\end{tabular}}
\caption{Example~2: experimental convergence using homogeneous
Dirichlet boundary conditions on a 3D domain $\O$
and using the viscosity function $\nu_c$.}
\label{taexam2}
\end{center}
\end{table}

\begin{figure}[!t]
\begin{center}
\includegraphics[height=0.35\textwidth]{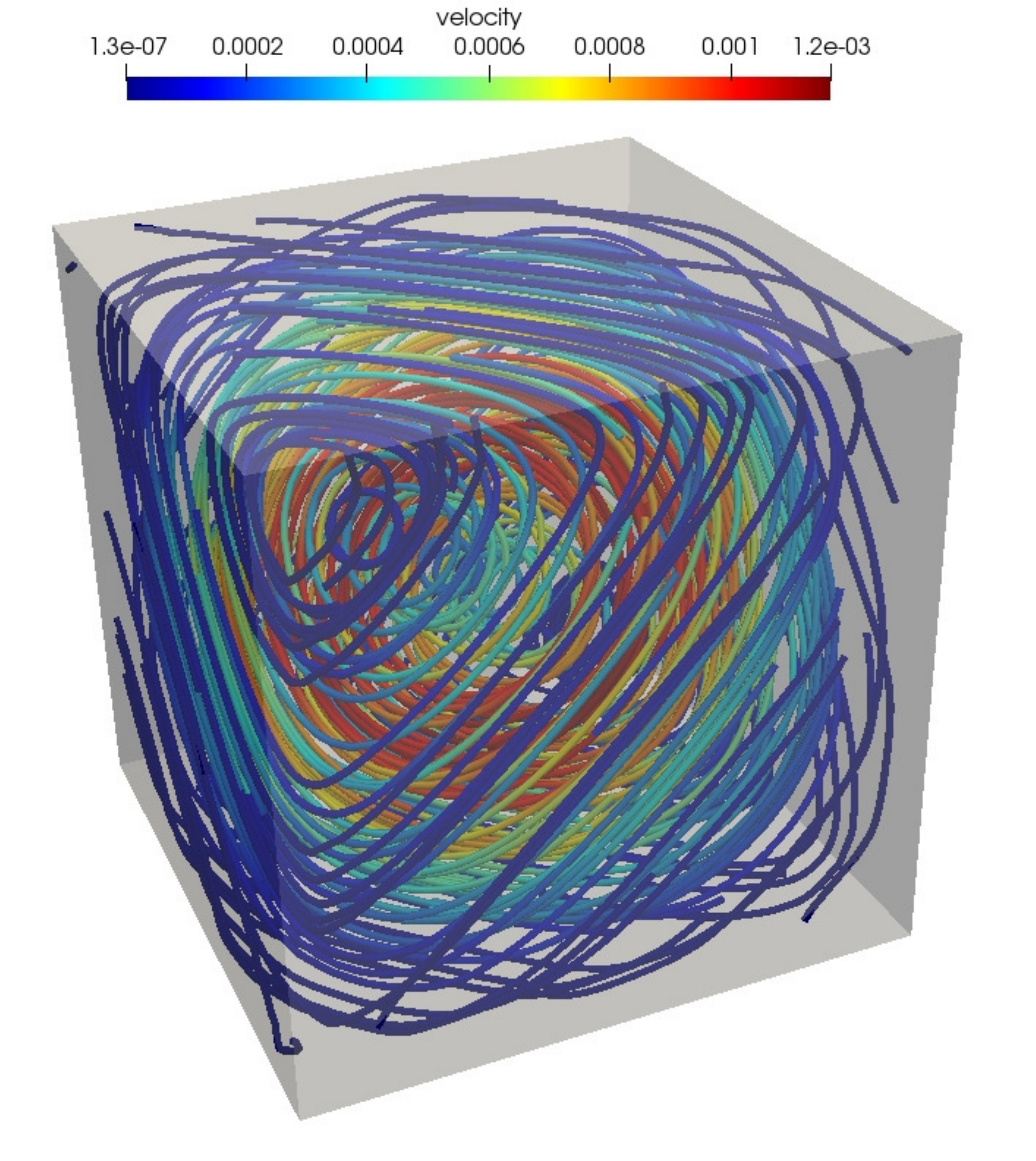}
\includegraphics[height=0.35\textwidth]{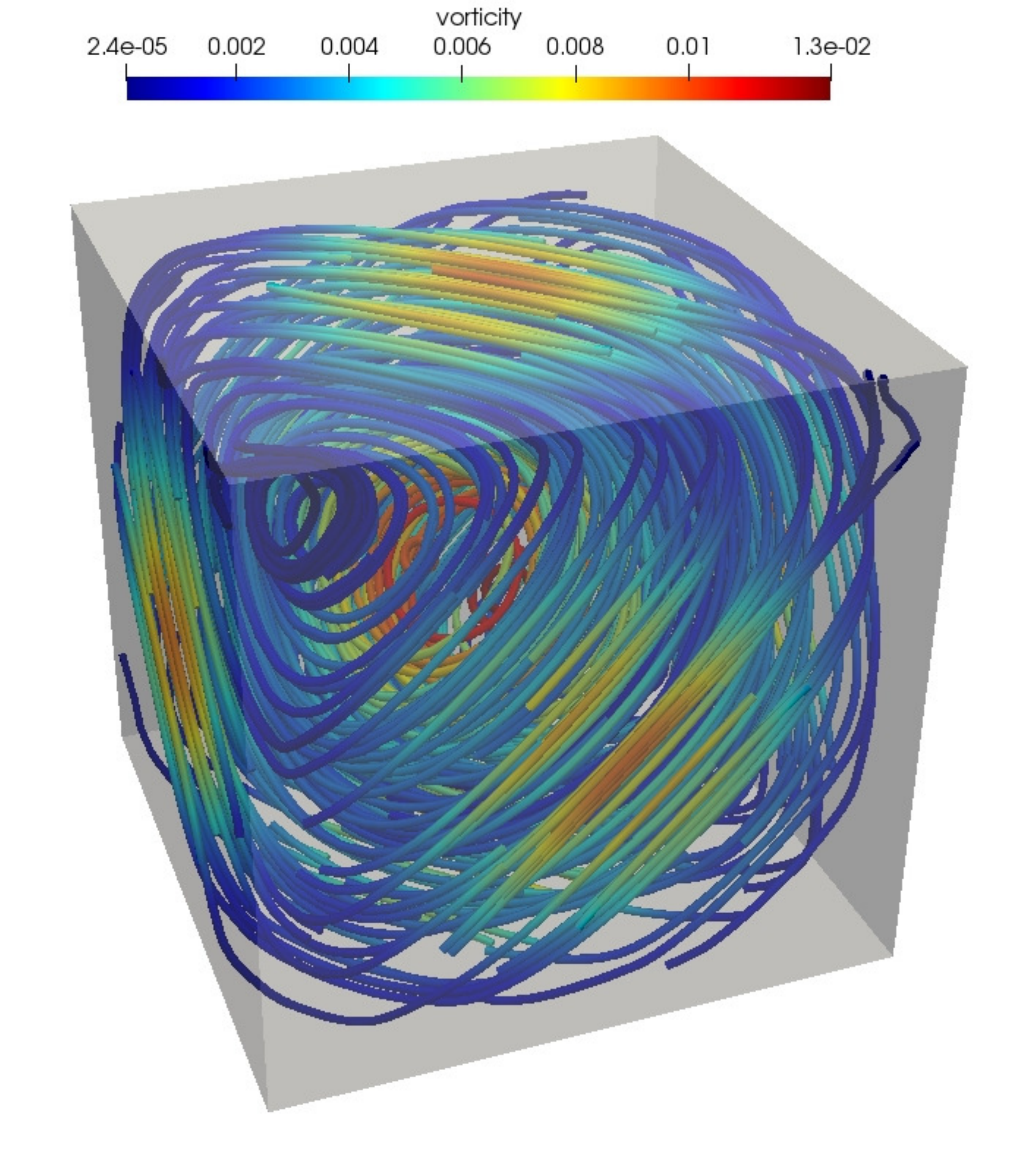}
\includegraphics[height=0.35\textwidth]{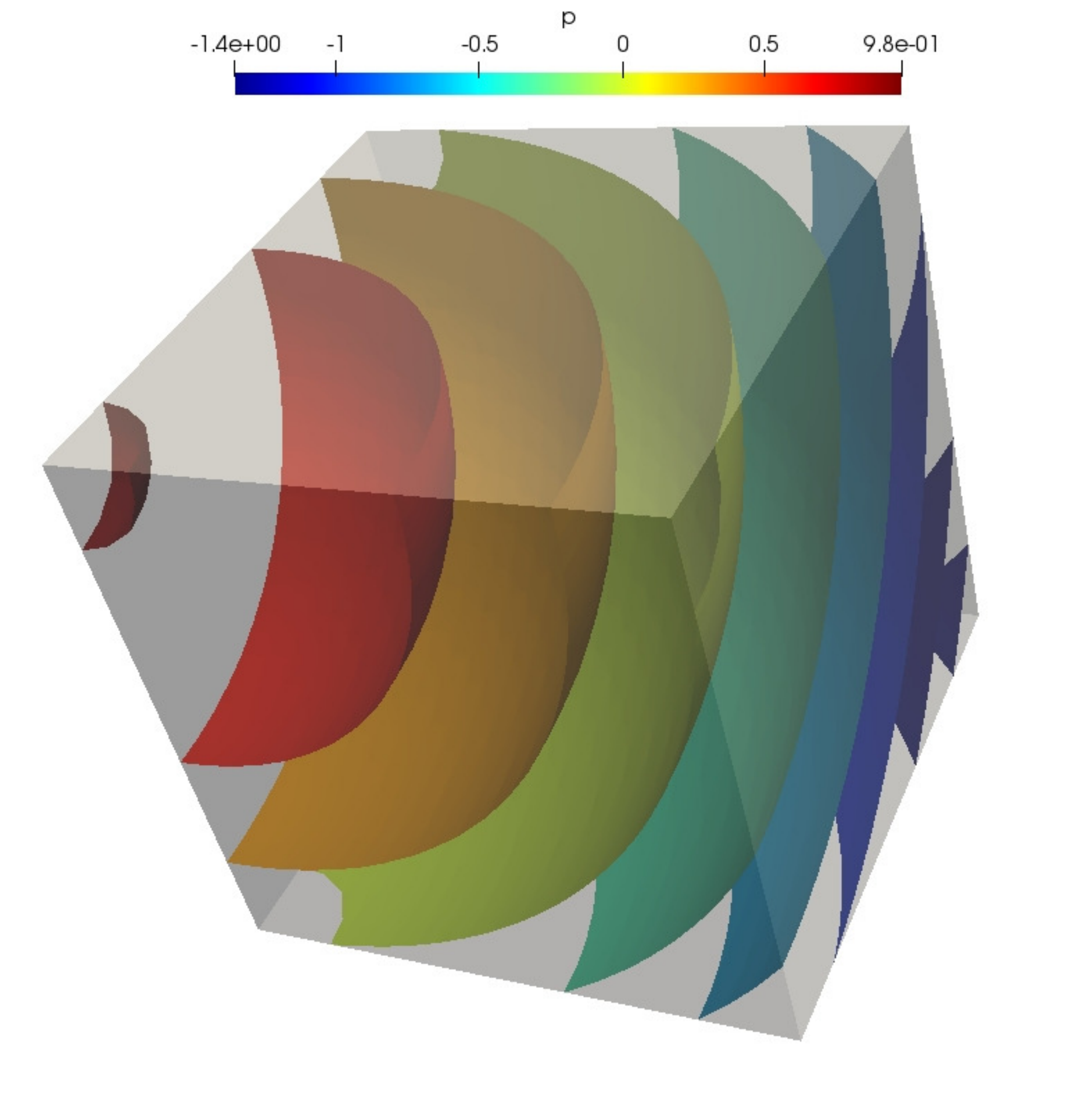}
\caption{Example~2: Approximate solutions computed using the MINI-element. Velocity streamlines (left) vorticity streamlines (centre) and 
pressure distribution (right).}
\label{vpvfigures3D}
\end{center}
\end{figure}

\subsection{Example 3: A posteriori error estimates and adaptive mesh refinement}
In this numerical test, we test the efficiency of the a posteriori error
estimator \eqref{globalestimator} and applying mesh refinement according
to the local value of the indicator.
In this case, the convergence rates are obtained by replacing the
expression $\log(h/\hat{h})$ appearing in the computation
of \eqref{rates} by $-\frac{1}{2}\log(N/\hat{N})$, where $N$ and $\hat{N}$
denote the corresponding degrees of freedom of each triangulation.

\begin{figure}[!t]
\begin{center}
\includegraphics[width=0.24\textwidth]{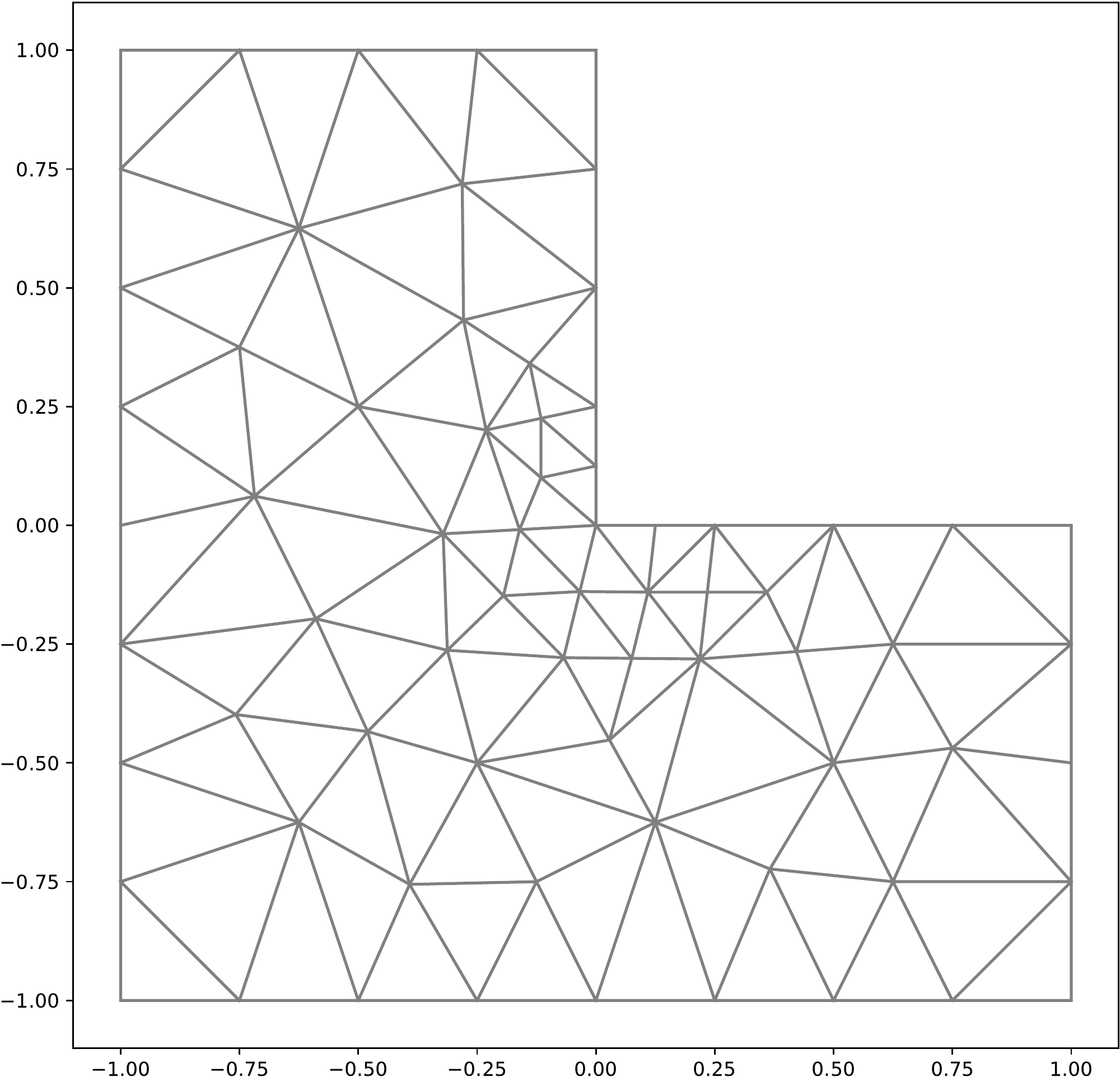}
\includegraphics[width=0.24\textwidth]{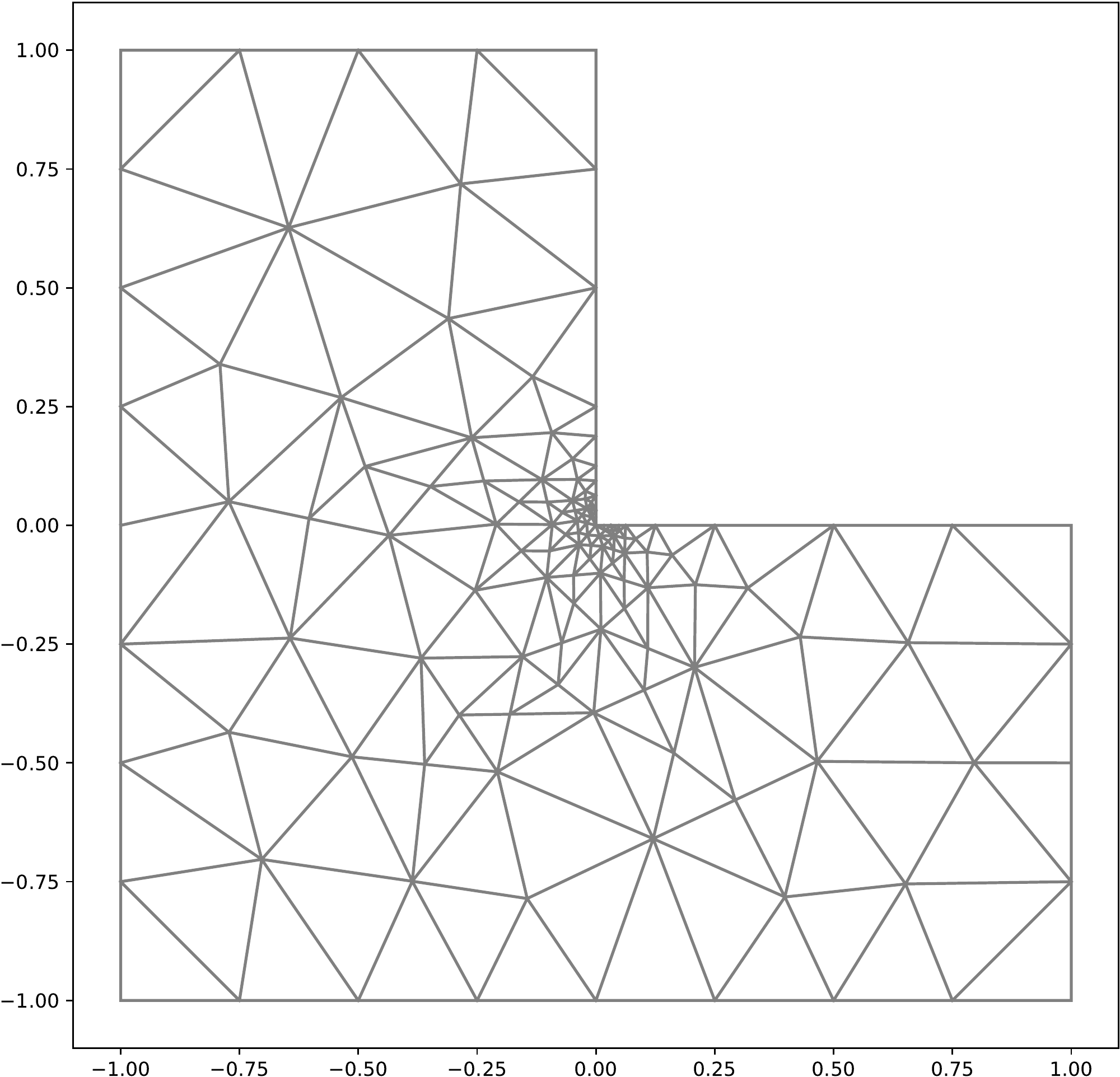}
\includegraphics[width=0.24\textwidth]{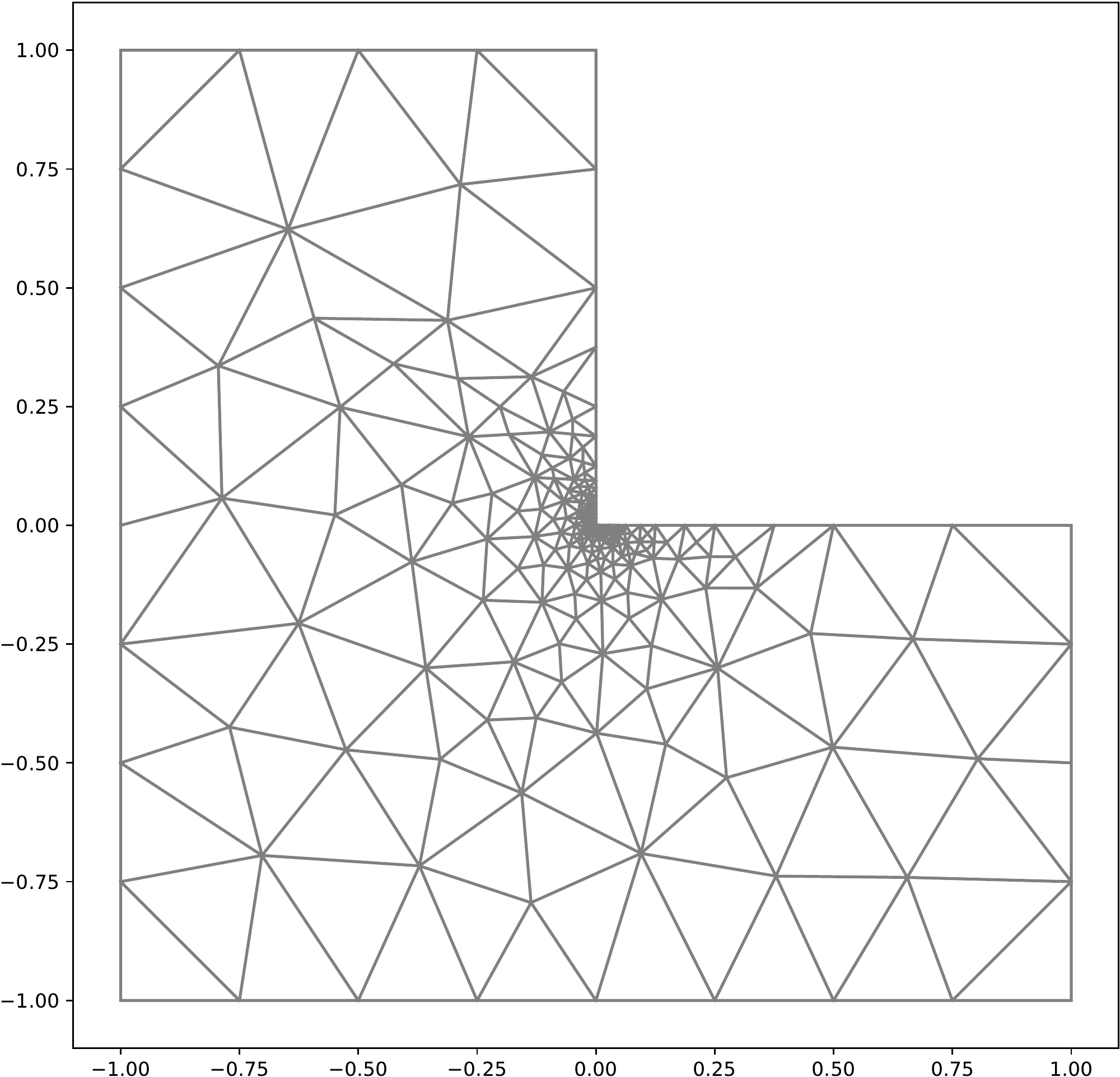}
\includegraphics[width=0.24\textwidth]{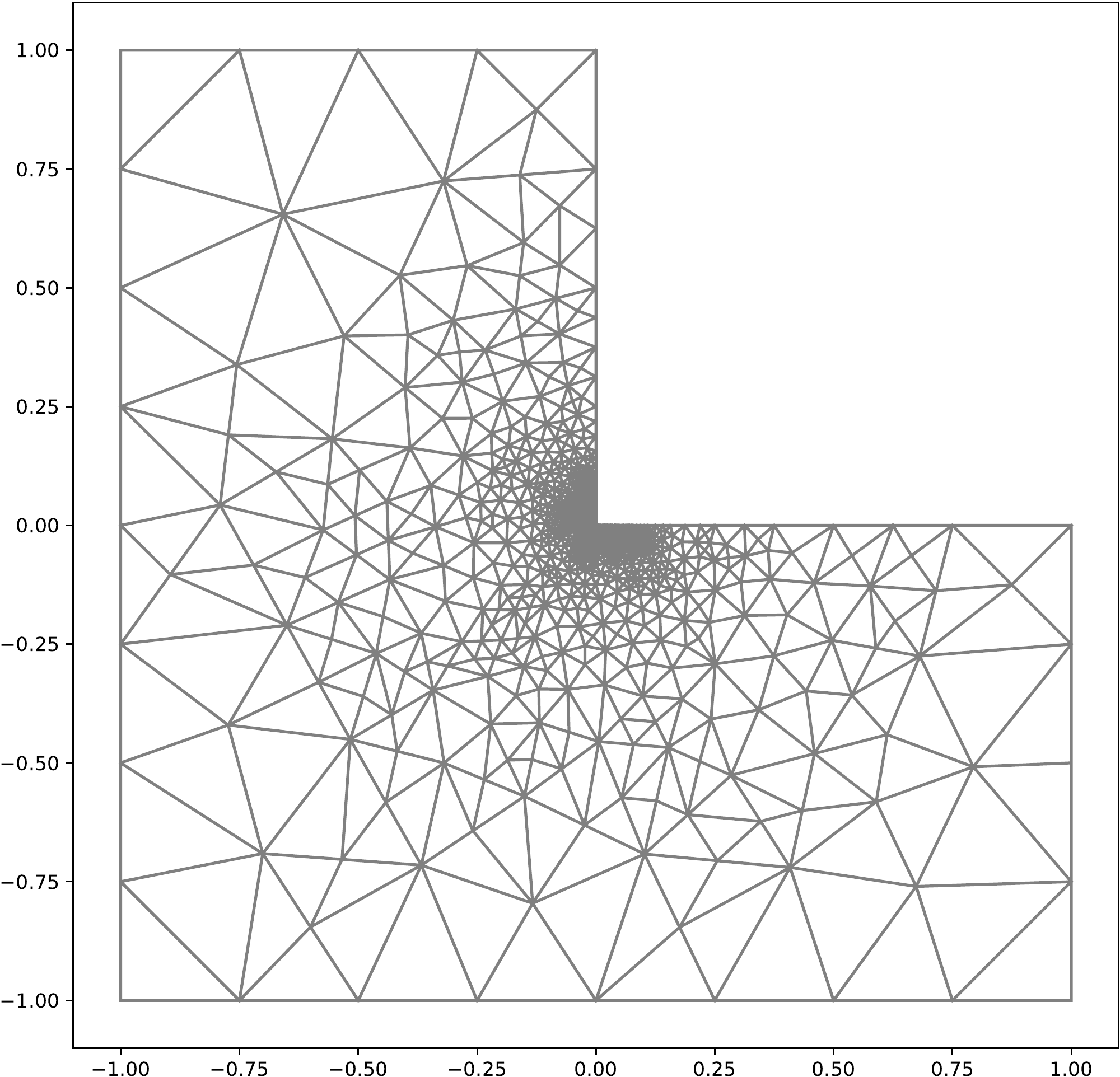}
\caption{Example~3: Snapshots of four grids, $\mathcal{T}_h^1$,
$\mathcal{T}_h^4$, $\mathcal{T}_h^6$, 
$\mathcal{T}_h^{10}$, adaptively refined according
to the a posteriori error indicator defined in \eqref{globalestimator}.}
\label{fig:adaptmesh}
\end{center}
\end{figure}

Now, we recall the definition of the so-called effectivity index
as the ratio between the total error and the global error estimator, i.e.,
\[{\tt e}(\bu,\omega,p):= \Big\{[e(\bu)]^2+[e(\omega)]^2+[e(p)]^2\Big\}^{1/2}, 
\qquad {\tt eff}(\Theta):= \frac{{\tt e}(\bu,\omega,p)}{\Theta}.\]

We will employ the family of finite elements
introduced in Section~\ref{tayhood} for $k=1$,
namely piecewise quadratic and continuous elements
for velocity and piecewise
linear and continuous elements for 
vorticity and pressure fields.

The computational domain is the nonconvex L-shaped
domain $\O=(-1,1)^2\setminus(0,1)^2$, where 
problem \eqref{eq:momentum1}-\eqref{eq:mass1} admits the following exact solution 
\begin{gather*}
p(x,y):=\dfrac{1-x^2-y^2}{(x-0.025)^2+(y-0.025)^2}-12.742942014/3,\\
\varphi(x,y)=x^2(1-x)^2y^2(1-y)^2\exp(-50((x-0.025)^2+(y-0.025)^2)),\quad 
\bu=\curl\varphi, \quad 
\omega=\rot\bu,\end{gather*}
which satisfy the incompressibility constraint as well as the
boundary conditions. Convective velocity, viscosity, and other parameters are taken as 
\begin{gather*} \bbeta=\bu,\quad   
\nu_d(x,y)=\nu_0+\dfrac{721}{16}(\nu_1-\nu_0)x^2(1-x)y^2(1-y), \quad 
\nu_0 = 0.1, \quad \nu_1 = 1,\\
\nu_e(x,y)=\nu_0+(\nu_1-\nu_0)\exp(-10^{12}((x-0.5)^{10}+(y-0.5)^{10})), 
\quad \kappa_1=\frac{2}{3}\nu_0, \quad \kappa_2=\frac{\nu_0}{2}, \quad \sigma = 10.\end{gather*}
Pressure is singular near the reentrant corner
of the domain and so we expect hindered convergence of the
approximations when a uniform (or quasi-uniform) mesh refinement is applied.
In contrast, if we apply the following adaptive mesh refinement procedure from \cite{verfuerth96}:
\begin{enumerate}
\item[1)]
Start with a coarse mesh $\mathcal{T}_h$.
\item[2)]
Solve the discrete problem \eqref{probform2d} for the current mesh
$\mathcal{T}_h$.
\item[3)]
Compute $\Theta_T:=\Theta$ for each triangle
$T\,\in\,\mathcal{T}_h$.
\item[4)]
Check the stopping criterion and decide whether to finish or go to
next step.
\item[5)]
Use {\it blue-green} refinement on those $T'\in\mathcal{T}_h$ whose
indicator $\Theta_{T'}$ satisfies
\[
\Theta_{T'} \,\ge\,\frac{1}{2}\,\max_{T\in \cT_h} \left\{
\Theta_{T}:\,\, T\in\mathcal{T}_h\,\right\}\,.
\]
\item[6)]
Define resulting meshes as current meshes $\mathcal{T}_h$ and
$\mathcal T_h$, and go to step 2,
\end{enumerate}
we expect a recovering of the optimal convergence rates.
In fact, this can be observed from the bottom rows
of Tables~\ref{taposter1} and \ref{taposter2}, for both
 $\nu_d$ and $\nu_e$, respectively.
Moreover,  the efficiency indexes are around 1
for both viscosities.
The resulting meshes after a few adaptation steps
are reported in Figure~\ref{fig:adaptmesh}, showing the expected 
refinement near the reentrant corner.

\begin{table}[!t]
\begin{center}
{\small\begin{tabular}{@{}rcrcrcrcc@{}}
\toprule
$N$&$\nnorm{\bu-\bu_h}_{1,\O}$& $r(\bu)$ &$\Vert \omega-\omega_h\Vert_{0,\O}$
&$r(\omega)$& $\Vert p-p_h\Vert_{0,\O}$ & $r(p)$&${\tt eff}(\Theta)$\\ \midrule
661 &  49.68 & -- & 8.821 & -- & 6.685 & -- & 1.133   \\
999 &  32.37 & 2.07 & 5.069 & 2.68 & 3.985 & 2.50 &    1.157  \\
1241 & 15.46 & 6.81 & 2.104 & 8.10 & 1.846 & 7.09 &    1.144 \\
1881 & 9.058 & 2.57 & 1.396 & 1.97 & 1.057 & 2.68 &    1.098  \\
2103 & 7.178 & 4.17 & 0.907 & 7.72 & 0.828 & 4.36 &  1.135 \\
2621 & 5.645 & 2.18 & 0.754 & 1.67 & 0.655 & 2.12 &  1.120  \\
3851 & 3.647 & 2.27 & 0.454 & 2.63 & 0.418 & 2.33 &   1.168\\
4267 & 3.243 & 2.29 & 0.401 & 2.46 & 0.365 & 2.61 &  1.156\\
5271 & 2.687 & 1.77 & 0.298 & 2.76 & 0.294 & 2.03 &  1.143\\
7819 &  1.754 & 2.16 & 0.194 & 2.18 & 0.191 & 2.22 & 1.155\\ \bottomrule
\end{tabular}}
\caption{Example~3: Convergence history and 
effectivity indexes for the 
method introduced in Section~\ref{tayhood},
computed on a sequence of adaptively refined 
triangulations of the L-shaped domain and using viscosity $\nu_d$.}\label{taposter1}
\end{center}
\end{table}

 \begin{table}[!t]
\begin{center}
{\small\begin{tabular}{@{}rcrcrcrcc@{}}
\toprule
$N$&$\nnorm{\bu-\bu_h}_{1,\O}$& $r(\bu)$ &$\Vert \omega-\omega_h\Vert_{0,\O}$
&$r(\omega)$& $\Vert p-p_h\Vert_{0,\O}$ & $r(p)$&${\tt eff}(\Theta)$\\ \midrule
661  & 49.73 & -- & 8.842 & -- & 6.681 & -- & 1.132\\
999  & 32.39 & 2.07& 5.081 & 2.68 & 3.980 & 2.50 & 1.155\\
1241  & 15.50 & 6.79 & 2.122 & 8.05 & 1.838 & 7.12 & 1.138\\
1881  & 9.087 & 2.56 & 1.401 & 1.99 & 1.039 & 2.74 & 1.085\\
2103  & 7.213 & 4.14 & 0.914 & 7.65 & 0.806 & 4.55 & 1.114\\
2589  & 5.683 & 2.29 & 0.759 & 1.78 & 0.633 & 2.32 &  1.112\\
3771  & 3.734 & 2.23 & 0.461 & 2.64 & 0.406 & 2.35 & 1.113\\
5161  & 2.674 & 2.12 & 0.307 & 2.58 & 0.287 & 2.20 & 1.108\\
6867  & 1.946 & 2.22 & 0.207 & 2.77 & 0.205 & 2.36 & 1.116\\
9887  & 1.346 & 2.02 & 0.128 & 2.60 & 0.138 & 2.16 & 1.119\\ \bottomrule
\end{tabular}}
\caption{Example~3: Convergence history and 
effectivity indexes for the 
 method introduced in Section~\ref{tayhood},
computed on a sequence of adaptively refined 
triangulations of the L-shaped domain and using viscosity $\nu_e$.}\label{taposter2}
\end{center}
\end{table}

\subsection{Example 4: Steady blood flow in aortic arch} 
We finalise the set of examples with a simple simulation of pseudo-stationary blood flow in an aorta. The patient-specific geometry \cite{marchandise11,mer1} has one inlet (a segment that connects with the pre-aortic root coming from the aortic valve in the heart) and four outlets (the left common carotid artery, the left subclavian artery, the innominate artery, and the larger descending aorta). On the inlet we impose a Poiseuille profile of magnitude 4, on the vessel walls we set no-slip conditions, and on the remaining boundaries we set zero normal stresses (more physiologically relevant boundary conditions can be considered following, e.g., \cite{dob,formaggia}). The initial unstructured mesh has 46352 tetrahedral elements. 
The synthetic variable viscosity field is a smooth exponential function $\nu = \nu_0 + (\nu_1-\nu_0)\exp(-10^3[(x-0.1)^6+y-0.5)^6+(z-0.5)^6])$ with $\nu_0 = 10^{-3},\nu_1=10$ that entails an average Reynolds number of approximately 60 (computed using the inlet diameter and maximal inlet velocity), while the convecting velocity is computed as the solution of a preliminary Stokes problem (on the initial coarse mesh), and we prescribe $\sigma = 1000$ and $\ff=\sigma\bbeta$. Then we compute numerical solutions of the Oseen problem and 
 apply four steps of adaptive mesh refinement using a 3D version of the estimator \eqref{globalestimator} and the algorithm 
 described in the previous example. 
The results are portrayed in Figure~\ref{fig:aorta}, plotting pressure distribution, velocity streamlines, vorticity, and a sample of the resulting adaptive mesh which shows more refinement near the boundaries of the descending aorta. For this test we have used a conforming approximation of vorticity. 

 \begin{figure}[!t]
\begin{center}
\includegraphics[width=0.325\textwidth]{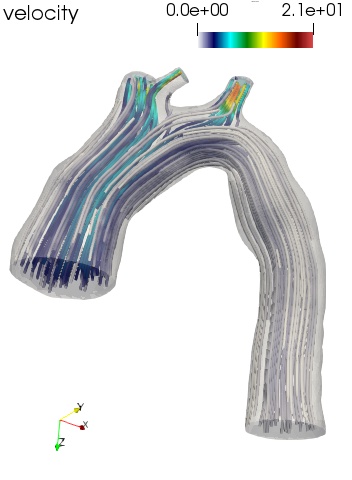}
\includegraphics[width=0.325\textwidth]{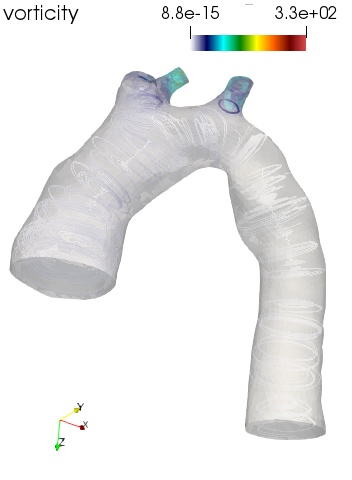}
\includegraphics[width=0.325\textwidth]{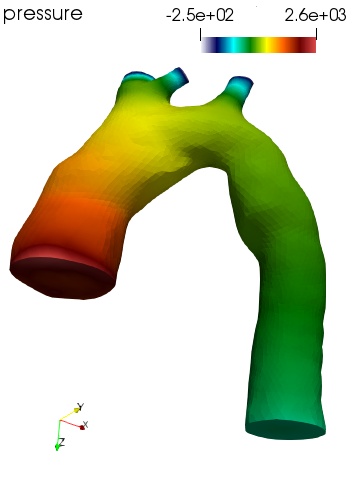}\\
\includegraphics[width=0.325\textwidth]{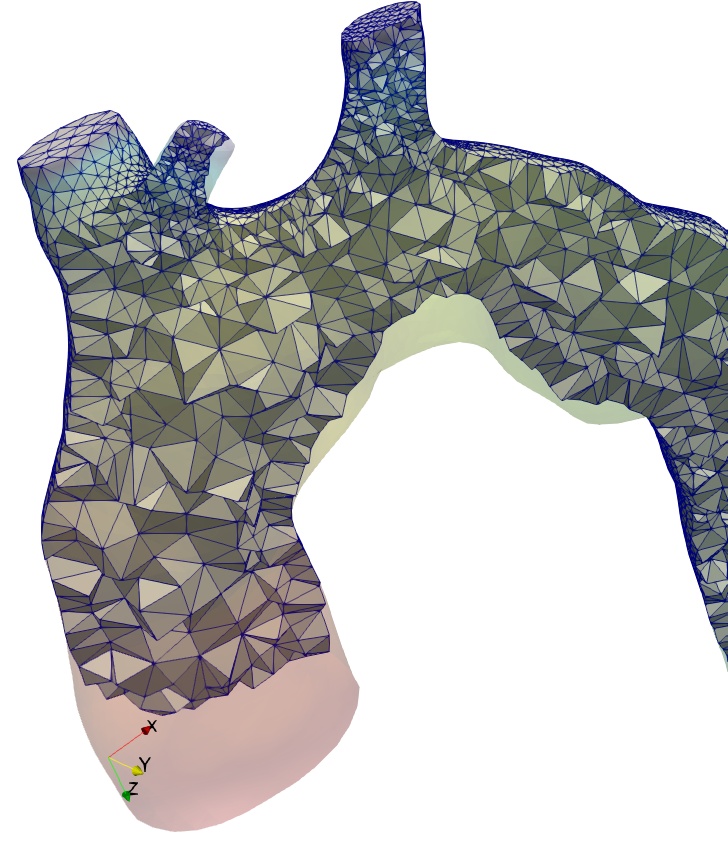}
\includegraphics[width=0.325\textwidth]{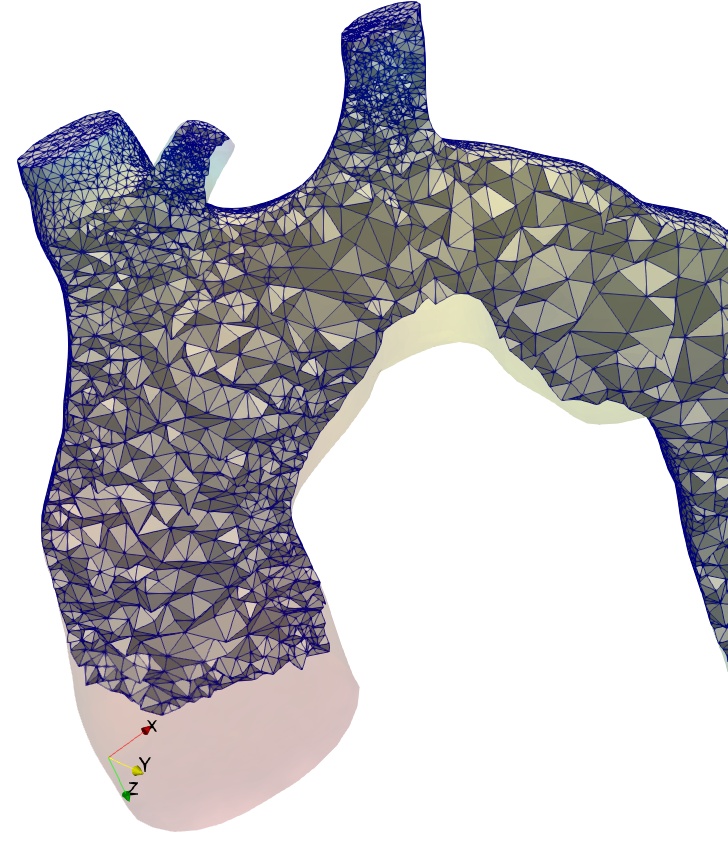}
\includegraphics[width=0.325\textwidth]{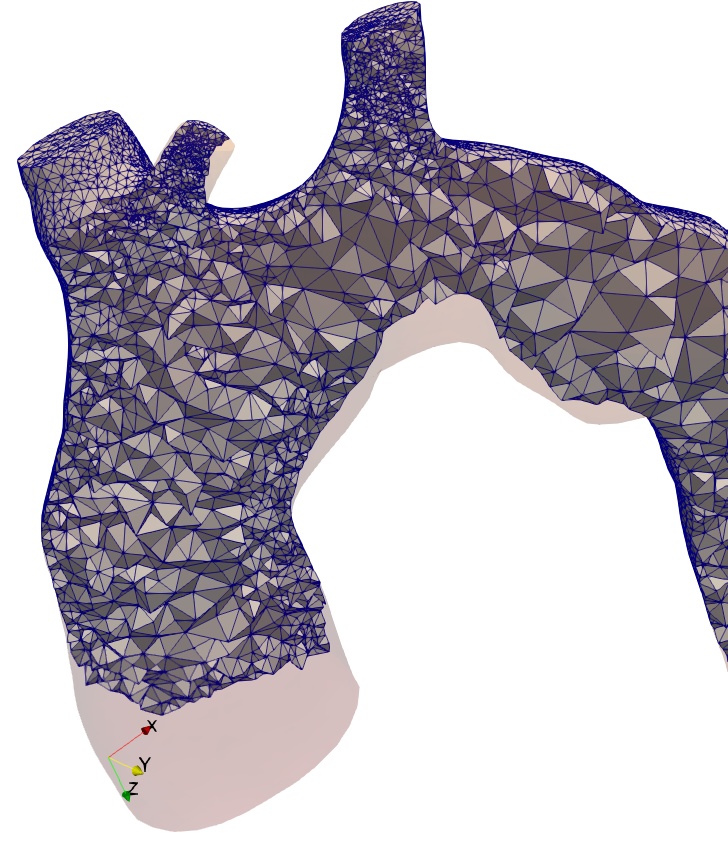}
\caption{Example~4: Simulation of stationary blood flow in an aortic arch. Approximate velocity, vorticity, and pressure (top panels), and 
samples of adaptive mesh after one, two and three refinement steps, and visualising a cut focusing on the boundaries (bottom row).}
\label{fig:aorta}
\end{center}
\end{figure}


\end{document}